\documentclass[11pt,a4paper, twoside]{article}
\usepackage[english]{babel}
\usepackage{amsmath}
\usepackage{amsthm}
\usepackage{amsfonts}
\usepackage{amssymb}
\usepackage[english]{varioref}
\usepackage{graphicx}
\usepackage{amsthm} 
\usepackage{amsmath}
\usepackage{eucal}
\usepackage{setspace}
\usepackage{version}
\numberwithin{equation}{section}

\newtheorem{theorem}{Theorem}[section]   %Theorem 1.1, Lemma 1.2  

\newtheorem{lemma}[theorem]{Lemma}

\theoremstyle{definition}

\usepackage{mathtools}

\usepackage{tocloft}

    \DeclareUnicodeCharacter{2212}{-}

\title{Blow up criteria for a fluid dynamical model arising in astrophysics}

\author{
\large{\bf{Donatella Donatelli} and \bf{Lorenzo Pescatore}}\\
\normalsize Department of Information Engineering, Computer Science and Mathematics \\
\normalsize University of L'Aquila \\
\normalsize 67100 L'Aquila, Italy. \\
\normalsize {donatella.donatelli@univaq.it}, {lorenzo.pescatore@graduate.univaq.it}
}
\date{}
\begin{document}
\maketitle

\begin{abstract}
In this paper we deal with the existence of local strong solution for  a perfect compressible viscous fluid, heat conductive and self gravitating, coupled with a first order kinetics used in  astrophysical hydrodynamical models. In our setting the vacuum is allowed  and as a byproduct of the existence result we get a blow-up criterion for the local strong solution. Moreover we prove a blow-up criterion for the local strong solutions in terms of the  velocity gradient, the mass fraction  gradient and the temperature similar to the well known Beale-Kato-Majda criterion for ideal incompressible flows.

\end{abstract}
\smallskip


\section{Introduction}
In this paper we consider a perfect compressible viscous fluid,  heat conductive and self gravitating, coupled with a first order kinetics.
\begin{equation}\label{pb1}
		\begin{cases}
\rho_t+ \operatorname{div}(\rho u)=0 \\
c_v((\rho \theta)_t+\operatorname{div}(\rho \theta u))-k\Delta \theta+ p\operatorname{div}u=Q(\nabla u)+ \rho h\\
(\rho u)_{t}+ \operatorname{div} (\rho u \otimes u)+ Lu+ \nabla p= \rho f\\
(\rho Z)_{t}+\operatorname{div}(\rho uZ))=-k \phi (\theta) \rho Z +\operatorname{div}(D \nabla Z),\\
\Delta \Phi=4\pi G  \rho, \\
\end{cases}
\end{equation}
\\
 where $t\geq 0$, $x\in \Omega\subset\mathbb{R}^{3}, \, \Omega$ bounded,
 $$Q(\nabla u)=\frac{\mu}{2}|\nabla u+\nabla {u}^T|^2+\lambda (\operatorname{div}u)^2,\quad  Lu=-\mu \Delta u-(\lambda+\mu)\nabla\operatorname{div}u.$$
\\
The system \eqref{pb1} is of interest since in astrophysical literature it is a reasonable model to describe the evolution of stars \cite{Cha1}, \cite{Cha2}, \cite{Kipp}, \cite{Led}.
Hence we denote by $\rho=\rho(x,t), \theta=\theta(x,t),  u=u(x,t)$ and $ Z=Z(x,t)$ the unknown density,  temperature,  velocity and mass fraction of the reactant.  It is well known that in the case of a star,  a lot of nuclear reactions take place giving rise to "burning" and combustion phenomena.
\\
In order to describe the gravitational phenomena and the combustion reactions we deal with the following profiles for the external force and heat source
$$f=-\nabla \Phi,\quad  h=qk\phi(\theta)Z, $$
where $\Phi$ is the gravitational potential which solves the following elliptic partial differential equation
$$ \Delta \Phi=4\pi G  \rho,$$
with the following representation formula 
\begin{equation*} \label{RepPoisson}
\Phi(x,t)=-G\int_{\Omega} \dfrac{\rho(z,t)}{|x-z|}dz.
\end{equation*}
\\
Concerning the heat flow $h$, for the Arrhenius law $\phi(\theta)$, the most used function is the following 

\begin{equation*} 
\begin{split}
\phi(\theta)= \begin{cases} e^{-\frac{E}{\theta}} \quad \quad \quad & \theta >0 \\ 0  \quad \quad \quad \quad & \theta \le 0,
\end{cases}
\end{split}
\end{equation*}
where $E$ is a positive constant representing the activation energy.
\\
In this paper we consider the so called modified Arrhenius law 

\begin{equation*} 
\begin{split}
\phi(\theta)= \begin{cases} \theta^{\alpha} e^{-\frac{E}{\theta}} \quad \quad \quad & \theta >0 \\ 0  \quad \quad \quad \quad & \theta \le 0,
\end{cases}
\end{split}
\end{equation*}
which is more appropriate in order to describe astrophysical phenomena as discussed in \cite{Ducomet}.
Our choice for the rate constant is $\alpha=\frac{1}{2}$ which is appropriate to deal with combustion processes.
We point out that in this paper the regularity of the external force $f$ and heat source $h$ is not a priori assigned,  indeed they strongly depends on the solution itself.  This represents one of the main issues in the computations for the local existence and the blow up results and it is a significant difference with respect to previous results in the fluid dynamic literature. \\
Polytropic fluids are characterized by the following constitutive law
$$ p=R\rho\theta, $$
while the viscosity and thermal coefficients satisfy the following natural restrictions
\[
\mu>0,\quad 3\lambda+2\mu\geq 0, \quad c_v >0,\quad  R> 0,\quad q\geq 0,\quad  k\geq 0.
\]
In the absence of vacuum,  a local existence theory regarding local strong solutions for heat conductive fluids has been developed by Nash \cite{Nash} and Itaya \cite{Ita} while uniqueness is discussed by Serrin in \cite{Serrin}.  In our analysis the vacuum is allowed. The key difference is that energy and momentum equations loose their parabolicity in the presence of vacuum.  This lack of regularity needs to be compensated by a compatibility condition which was first introduced by Straskaba in \cite{Strask} using a fixed point argument. 
In the vacuum case,  some results has been obtained in a series of papers by Cho and collaborators \cite{cho2004},\cite{cho2006},\cite{Cho 2006 bis}. They proved a local existence results for strong solution and a blow up criteria for both isoentropic or polytropic flows by using an iteration argument in the framework of Sobolev spaces for bounded or unbounded domains. 
Later on a more refined blow up criteria were proved,  for isoentropic fluids Huang and Xin \cite{huang2010} proved a blow up criteria of the Beale Kato Majda type \cite{beale},  involving only $\nabla u$ and by assuming a weaker compatibility condition on the initial data than \cite{Cho 2006 bis}.
Moreover in \cite{Sun} Sun et al.  established a blow up criterion in terms of the upper bound of the density,  in the case of zero external forces.  This results keeps the analogy with the incompressible case,  where thanks to the dualism between pressure and velocity, high regularity on the pressure corresponds to the smoothness of the Leray weak solutions.  A blow up criteria involving $\operatorname{div}u$ was proved by Du. et al. in \cite{du} which is relevant in terms of the Leray-Hopf decomposition, while Fan et collaborators in \cite{fan} generalized the result of Huang in the case of heat conductive flows,  by adding a term depending on the temperature. 
For system (\ref{pb1}),  some results on the existence of global weak solutions has been obtained in a paper by Donatelli and Trivisa in \cite{Donatelli} in the case of zero external forces, 
They proved the existence of  global weak solutions for an initial boundary value problem with large initial data by using weak convergence methods, compactness and interpolation arguments. 

\subsection{Main results}
Our analysis is focussed on the existence of strong solutions with vacuum,  in the framework of Sobolev spaces.
Hence we assume the following initial conditions
\begin{equation}\label{ic1}
\rho_0\ge 0, \, \rho_0\in H^{1}\cap W^{1,q},  \, 0\le Z_0 \le 1,\,  (\theta_0,u_0,Z_0)\in H_{0}^1\cap H^2, \, q\in (3,6]
\end{equation}
boundary conditions,
\begin{equation}
(\theta, u,Z)=(0, 0,0) \quad \text{on} \quad (0, T)\times \partial\Omega
\end{equation}
and compatibility conditions,
\begin{equation}\label{cc1d}
-k\Delta \theta_0-Q(\nabla u_0)=\rho_{0}^\frac{1}{2}g_1 \quad \text{and} \quad Lu_0+\nabla p_0=\rho_{0}^\frac{1}{2}g _2\quad \text{in}\quad \Omega
\end{equation}
for some $(g_1,g_2)\in L^2,$  $p_0=R\rho_0\theta_0.$ 
\\
A stronger compatibility condition can be formally obtained by sending $t\rightarrow 0$ in the momentum and energy equation. Indeed by rewriting the momentum equation in the following form: 
$$ -\mu \Delta u-(\lambda+\mu) \nabla \operatorname{div}u+ \nabla p=\rho g$$ 
with $g=f-u_t-u\cdot \nabla u$
and by assuming $(\rho, u)$ smooth and by sending $t\rightarrow 0$ we get:
$$\mu \Delta u_0-(\lambda+\mu) \nabla \operatorname{div}u_0+ \nabla p_0=\rho_0 g_1$$ where $g_{1}=g_{|t=0}$. This condition turns out to be a stronger compatibility condition  with respect to \eqref{cc1d} and leads to some additional regularity on the solution as discussed in \cite{cho2006}. Despite this,  during the computations it turns out that (\ref{cc1d}) is the exact regularity needed because it represents the $L^2$ integrability of $\sqrt{\rho} u_t, \sqrt{\rho} \theta_t$ at $t=0.$ Moreover we point out that \eqref{cc1d} is also a necessary condition for the existence of strong solutions in the regularity class of Theorem \ref{Teo2.1}. Indeed by rewriting the momentum equation as
$$Lu(t)+\nabla {p(t)}=\rho(t)^\frac{1}{2} {G_1(t)}$$ with $G_1(t)=\rho(t)^\frac{1}{2} (f-u_t-u \cdot \nabla u)$ and since $G_1\in L^\infty(0, T_*; L^2)$ we have that there exists a sequence $ \{t_k\}, \, {t_k} \rightarrow 0$ such that $G_1(t_k)$ is a bounded in $L^2$ so  it is weakly convergent to a certain $g_1\in L^2.$
A similar argument holds for the energy equation. 
Note that whenever the initial density $\rho_0$ is bounded away from zero,  then condition \eqref{cc1d} is automatically satisfied for initial data with the regularity \eqref{ic1}. \\
Our first result, Theorem \ref{Teo2.1} extends the local existence result in \cite{Cho 2006 bis} by adding the time evolution equation for the mass fraction $Z.$ It gives us the
class of regularity for strong solutions and also uniqueness.  In order to prove this result,  we use an approximation technique considering first a linearized system where we were able to prove the well posedness of the problem and some useful bounds and then,  after defining properly the approximate solution,  we prove convergence to the local solution of the nonlinear original problem. 
We stress out that the dependence of the external force $f$ and heat source $h$ on the solution leads to a strongly coupled system in the iteration scheme arising a lot of new technical issues.
%It can be stated as follows

\begin{theorem}
\label{Teo2.1}
Assume that the data $(\rho_0, \theta_0, u_0, Z_0)$ satisfy the regularity condition $$\rho_0\ge 0,\quad \quad \rho_0\in H^{1}\cap W^{1,q}, \quad 0\le Z_0\le 1 \quad (\theta_0,u_0, Z_0)\in H_{0}^1\cap H^2,$$
and the compatibility condition
\begin{equation}\label{compat cond}
-k\Delta \theta_0-Q(\nabla u_0)=\rho_{0}^\frac{1}{2}g_1 \quad \text{and} \quad Lu_0+\nabla p_0=\rho_{0}^\frac{1}{2}g _2\quad in\quad \Omega,
\end{equation}
for some $(g_1,g_2)\in L^2,$ where $p_0=R\rho_0\theta_0.$
Then there exists a small time $T_*>0$ and a unique strong solution $(\rho, \theta, u, Z)$ to the initial value problem (\ref{pb1})-(\ref{cc1d}) such that
$$ \rho\in C([0,T_*];H^{1}\cap W^{1,q}),  \quad \rho_t\in C([0, T_*]; L^2\cap L^q), $$
$$(\theta,  u,  Z)\in C([0, T_*]; H_{0}^1\cap H^2)\cap L^2(0, T_*; W^{2,q}),$$
\begin{equation}\label{crf}
(\theta_t,  u_t,  Z_t)\in L^2(0, T_*; H_{0}^1) \ \text{and}\ (\sqrt{\rho}\theta_t,   \sqrt{\rho}u_t, \sqrt{\rho}Z_t) \in L^\infty (0, T_*; L^2).
\end{equation}
\end{theorem}
After having established a local existence result for strong solution,  we investigate the possibility to extend it from a local to a global one and related blow up criteria. 
Theorem \ref{firstblowup} is a generalisation of the blow up criterion for strong solution for the compressible Navier Stokes equations obtained in \cite{cho2004} by Cho and collaborators.  It is a consequence of the regularity estimates obtained in Theorem \ref{Teo2.1} and a contradiction argument.
Our main contribution concerns the estimates needed in order to apply the contradiction argument used in \cite{cho2004} to the temperature $\theta$ and the mass fraction $Z$.
\begin{theorem}
\label{firstblowup}
Let $(\rho, \theta, u, Z)$ be the local strong solution of the initial boundary value problem (\ref{pb1})-(\ref{cc1d}) with regularity (\ref{crf}).  If $T_*$ is the maximal existence time and $T_*<T,$ then $$\limsup_{t \rightarrow T_*}( |\rho(t)|_{H^{1} \cap W^{1,q}}+|u(t)|_{H_{0}^1}+|\theta(t)|_{H_{0}^1}+|Z(t)|_{H_{0}^1})=\infty.$$
\end{theorem}

Finally, we prove a more refined blow up criterion of Beale Kato Majda type, generalizing \cite{fan}. Indeed it turns out that $\nabla u$ is one of the responsible of the breakdown of smooth solutions. 
This result shows a relevant difference concerning the blow up analysis for compressible and incompressible fluids. While in the incompressible case the smoothness of classical solutions is controlled by $\omega=\operatorname{curl}u,$ the antisymmetric part of $\nabla u$, in the compressible case we must take into account also the contribution due to its symmetric part. 
In order to apply the contradiction argument used in \cite{fan}, we modify their blow up result by adding $\| \nabla Z\|_{L^1(0, T; L^\infty)}$ which is needed in order to get more regularity on the mass fraction $Z$  and allow us to extend the solution up to the maximal existence time $T_*$ and to get our contradiction, see Lemma \ref{Lemma 2.8 cap4},  Lemma \ref{Lemma 2.9 cap4}.
For this last result we will consider the framework of \cite{fan} where they required the following assumption on the shear and bulk viscosity coefficients: 
\begin{equation}\label{visk2}
\mu>\frac{1}{7}\lambda
\end{equation}
which is needed in order to have an higher order estimate on the velocity $u$,  see Lemma \ref{Lemma 2.3 cap4}.
\begin{theorem}
\label{scndblowup}
Let $(\rho, \theta, u, Z)$ be the local strong solution of the initial boundary value problem (\ref{pb1})-(\ref{cc1d}) with regularity (\ref{crf}).  If $T_*$ is the maximal existence time and $T_*<T,$ then 
\begin{equation}
\lim_{T \rightarrow T_*} (\|\theta\|_{L^\infty(0, T, L^\infty)}+\| \nabla u\|_{L^1(0, T; L^\infty)}+\| \nabla Z\|_{L^1(0, T; L^\infty)} =\infty 
\end{equation}
provided that (\ref{visk2}) is satisfied.
\end{theorem}
The outline of this paper is as follows.  In Section 2 we prove the well posedness of the linearized system of \eqref{pb1}.  Section 3 is dedicated to the proof of Theorem \ref{Teo2.1},  we define an inductive scheme and we construct an  approximating solution and prove the convergence to the solution of the nonlinear problem.  In Section 4 we perform the blow up analysis of the local strong solution.

\section{A priori estimates for the linearized system}
This section is devoted to the study of the local well posedness for the linearized system of \eqref{pb1}.   We first consider the case where the initial density is away from the vacuum,  then we prove the existence of strong solutions and we deduce some a priori estimates which are independent on the lower bound of the density $\delta.$ Finally we perform the limit $\delta\rightarrow 0 $ and we get the existence of the solution of the linearized problem in the vacuum case and some useful a priori estimates.  Uniqueness is also proved by a standard contradiction argument.

\subsection{Linearized system}
We consider the following linearized system of \eqref{pb1} in $ (0, T)\times \Omega $

\begin{equation}\label{densityl}
\rho_t+ \operatorname{div}(\rho v)=0
\end{equation}
\begin{equation}\label{momentuml}
(\rho u)_{t}+ \operatorname{div} (\rho v \otimes u)+ Lu+ \nabla p= \rho f
\end{equation}
\begin{equation}\label{temperaturel}
c_v((\rho \theta)_t+\operatorname{div}(\rho \theta v))-k\Delta \theta+ p\operatorname{div}v=Q(\nabla v)+ \rho h 
\end{equation}
\begin{equation}\label{Zl}
(\rho Z)_{t}+\operatorname{div}(\rho vZ))=-k \phi (\theta) \rho Z +\operatorname{div}(D \nabla Z)
\end{equation}
\begin{equation} \label{poissonl}
\Delta \Phi=4\pi G  \rho
\end{equation}
with initial data 
\begin{equation}\label{ic1l}
(\rho, \theta, u, Z)_{|{t=0}}=(\rho_0, \theta_0, u_0, Z_0) \quad in \quad \Omega
\end{equation}
and boundary conditions
\begin{equation}\label{bcl}
(\theta, u,Z)=(0,0,0) \quad on\quad (0, T)\times \partial\Omega
\end{equation}
\\
where $t\geq 0$, $x\in \Omega\subset\mathbb{R}^{3}$ and $v$ is a known vector field whose regularity will be defined later on. 
We start by solving the linear transport equation \eqref{densityl} by the characteristic method.  

\begin{lemma}
\label{lemma6}
Assume that $ \rho_0 \ge 0, \quad \rho_0 \in C_0 \quad \text{and}\quad  v\in L^\infty(0,T;H_0 ^1 \cap H^2) \cap L^2 (0,T;W^{2,q})$ with $q\in (3,6].$
Then there exists a unique weak solution $\rho \in  C ([0, T ]; C_0 ) $ to the linear hyperbolic problem \eqref{densityl},  \eqref{ic1l}.  Moreover the solution $\rho$  can be represented by
\begin{equation}\label{represden}
\rho(t, x)=\rho_0(U(0, t, x))\exp{ [-\int_{0}^{t} \operatorname{div}v(s, U(s, t, x))ds]}
\end{equation}
where $U\in  C([0, T ]\times  [0, T ]\times \Omega)$ is the solution to the initial value problem
\begin{equation}
\begin{cases}
U_t( s, t, x)= v(t, U(s, t, x)), \quad  0\le t\le T \\
U( s, s, x)= x, \quad 0\le s\le T
\end{cases}
\end{equation}
Assume in addition that $\rho_0 \in W^{1,r}$ for
some $r$ with $2\le r\le q.$ Then we also have $
\rho \in C([0,T];W^{1,r}) $ and  $\rho_t \in L^\infty( 0,T;L^r).$
\end{lemma} 
A detailed proof of  Lemma \ref{lemma6} can be found in \cite{Cho 2006 bis}.
\\
From now on we assume the following regularity:
\begin{equation}\label{reglin}
\begin{split}
 \rho_0 \ge 0, \rho_0 \in W^{1,r} \cap W^{1,q},  \quad 0\le Z_0\le 1 \quad (\theta_0,u_0, Z_0)\in H_0^1 \cap H^2,  
  \\ v \in C([0,T];H_0^1 \cap H^2)\cap L^2(0,T;W^{2,q}), v_t \in L^2(0,T;H_0^1) 
\end{split}
\end{equation}
for some constants $q$ and $r$ such that $ 2 \le r \le 3 < q \le 6. $
Then the global existence of a unique strong solution $(\rho , \theta, u, Z)$ to the linearized problem \eqref{densityl}-\eqref{bcl} can be proved by standard methods at least for the case that $\rho_0$ is bounded away from zero.
Hence,  in order to prove the existence of the solution for the linearized system,  we will first consider the case $\rho_0\ge \delta>0.$ Then we recover some estimates of the solution $(\rho^\delta,  \theta^\delta, u^\delta, Z^\delta)$ which are independent from the size of the domain $\Omega$ and from $\delta.$ 
We will get the existence of the solutions of the original linearized problem by the limit $\delta \rightarrow {0}.$
\\
The next Lemma guarantees the existence of solution of the linearized system in the case of a positive lower bound for the initial density.

\begin{lemma}
\label{lemma7}

Assume in addition to (\ref{reglin}) that $\rho_0\ge \delta$ in $\Omega$ for some constant $\delta > 0. $ Then there exists a unique strong solution $ (\rho , \theta,  u, Z)$ to the initial boundary value problem \eqref{densityl}–\eqref{bcl} such that:
\begin{equation}
\begin{split}
& \rho \in C([0,T];W^{1,r}\cap W^{1,q}) , \rho_t \in C([0,T];L^r \cap L^q),  \\ & (\theta, u,Z)\in C([0,T];H_0^1 \cap H^2)\cap L^2(0,T;W^{2,q}), \\ & (\theta_t,u_t, Z_t) \in C([0,T];L^2)\cap L^2(0,T;H_0^1),  \\ & (e_{tt},u_{tt}, Z_{tt})\in L^2(0,T;H^{−1}),  \\ & \rho \ge \delta\quad  [0,T] \times \Omega \quad  \delta >0, \quad 0 \leq Z \leq 1.
\end{split}
\end{equation}
\end{lemma}

\begin{proof}

The existence of a solution $\rho$ for the linear transport equation \eqref{densityl} is guaranteed by Lemma \ref{lemma6}.  Moreover,  by recalling the representation formula \eqref{represden} we have the following inequality:
\begin{equation}
\rho(t, x)=\inf \limits_{\Omega} \rho_0\exp{ [- C \int_{0}^{t} \|\nabla v\|_{W^{1,r}} ds]} \ge \delta >0,
\end{equation}
for $(t , x )\in  [0, T ] \times \overline{\Omega}$ from which we deduce that if we start with a positive initial density,  then $\rho$ remains positive for all times $t > 0 .$
This allows us to divide the equations \eqref{momentuml}-\eqref{temperaturel}-\eqref{Zl} by $\rho$ and we can rewrite them in the following way:

$$ \theta_t +v \cdot \nabla \theta +R\theta \operatorname{div}v−k\rho^{−1} \Delta \theta=\rho^{−1} Q(\nabla v)+h,$$

$$u_t +v\cdot \nabla u+\rho^{−1} Lu=f −R\rho^{−1} \nabla (\rho \theta),$$

$$Z_t+ v \cdot \nabla Z=-K \phi(\theta)Z+ \rho^{-1}D\Delta Z,$$
where the first and third equations are two linear parabolic equations,  while the second equation is a linear parabolic system for the velocity vector field $u.$
The proof of the existence of the solution for this problem can be proved by using a standard Galerkin approximation technique or with the method of continuity.
\end{proof}
Now we compute some a priori estimates independent of $\delta$ for the solutions of the linearized system.  
\\
We assume that the initial data satisfy \eqref{reglin} and we chose any fixed $c_0$ so that,
\begin{equation}\label{c0}
 c_0 \ge 1+ \| \rho_0\|_{H^1 \cap W^{1, q}}+ \| (\theta_0,  u_0, Z_0)\|_{H^1_{0} \cap H^2}+ \|( g_1, g_2)\|^2_{L^2},
 \end{equation}
where $3 < q \le 6,$ $g_1= {\rho_0}^{-\frac{1}{2}}(-k \Delta \theta_0- Q(\nabla{u_0}),$ $g_2= {\rho_0}^{-\frac{1}{2}}(L{u_0}+ \nabla {p_0}),$ 
$p_0=R\rho_0\theta_0.$
\\
Regarding the vector field $v$ we assume that $v(0)=u_0$ and
\begin{equation}\label{regularity v}
v \in C([0,T];H_0^1 \cap H^2)\cap L^2(0,T;W^{2,q}), v_t \in L^2(0,T;H_0^1).
\end{equation}
\subsection{Estimates for $\rho$}
We recall from \cite{Cho 2006 bis} that
\begin{equation}\label{repc0}
\|\rho(t)\|_{W^{1,r}\cap W^{1,q}} \le Cc_0\exp{ [C \int_{0}^{t} \|\nabla v\|_{H^1 \cap {W^{1,q}}} ds}]
\end{equation}
for $0\le t \le T_*.$
By using H\"older inequality we get:
$$\int_{0}^{t} \|\nabla v\|_{H^1 \cap {W^{1,q}}} ds \le t^\frac{1}{2}[\int_{0}^{t} {\|\nabla v\|}^2_{H^1 \cap {W^{1,q}}} ds]^{\frac{1}{2}} \le C(c_2t+(Cc_2)^\frac{1}{2}).$$
By \eqref{repc0} and \eqref{densityl} we end up with the following estimates for the density:
\begin{equation}\label{RHO W1Q}
\|\rho(t)\|_{W^{1,r}\cap W^{1, q}} \le Cc_0,
\end{equation}
\begin{equation}\label{RHO L3}
\| \rho_t(t)\|_{L^r\cap L^q}\le Cc^2_2
\end{equation}
for $ 0\le t \le \min( T_*, T_1),$ where $T_1=c^{-1}_{2}<1.$ Thanks to these estimates the density $\rho$ satisfies the following bounds:
\begin{equation}\label{limitrho}
C^{-1}\delta\le \rho(t,x)\le Cc_0,
\end{equation}
$ 0\le t \le \min( T_*, T_1),$ $x \in \overline{\Omega}.$
\subsection{Estimates for  Z}
Since $v$ is fixed,  in the equation \eqref{Zl} of the mass fraction $Z$ we apply the  parabolic regularity.  Since $0 \le Z_0 \le 1$ then, for all times $t$,  $Z_0 \in L^\infty$ and $$Z\in L^\infty ((0, T)\times \Omega).$$
We multiply  \eqref{Zl} by $Z$ and integrate over $\Omega$ and we get
\begin{equation}\label{Z1}
\dfrac{d}{dt} \int_{\Omega} \dfrac{1}{2}\rho Z^2 dx+ D \int_{\Omega} |\nabla Z|^2 dx=-k \int_{\Omega} \rho \phi(\theta) Z^2.
\end{equation}
Since the right hand side of \eqref{Z1} is non positive,  we get 

$$\dfrac{1}{2} \dfrac{d}{dt}{{\| \sqrt{\rho} Z \|}^2}_{L^2}+ D \int_{\Omega} {\| \nabla Z \|}^2 dx \le 0$$ 
from which we get 
\begin{equation}\sqrt{\rho}Z \in L^p(0, T; L^2(\Omega))  \ \ \ \text{ for any} \ p \ge 1 ,
\end{equation}
  \\ or equivalentely 
\begin{equation}\label{regularity for Z L2}
Z\in L^p(0, T; L^2(\Omega))  \ \ \ \text{ for any} \ p \ge 1 .
\end{equation}
 \\
Integrating \eqref{Z1} with respect to time we get:

\begin{equation}
Z\in L^2(0, T; H^1_0(\Omega)).
\end{equation}
\\
Then if we take the equation \eqref{Zl} and multiply by $Z_t$ we get:

\begin{equation}\label{Z2}
\begin{split}
& \dfrac{d}{dt}\int_{\Omega} \dfrac{D}{2}|\nabla Z|^2+\int_{\Omega} \rho |Z_t|^2  \\  & \le \int_{\Omega} |Z_t|\rho |v| |\nabla Z|+ \int_{\Omega}|Z_t| k \rho \phi(\theta)Z
\end{split}
\end{equation}
The first integral in the right hand side can be estimated in the following way,  
\begin{equation}\label{Z uno}
\begin{split}
& \int_{\Omega} |Z_t|\rho v |\nabla Z| \le C_1{\|\sqrt{\rho}Z_t\|}^2_{L^2}+C_2\| \nabla Z\|^2_{L^2}.
\end{split}
\end{equation}
In the second integral in the right hand side of  \eqref{Z2} we have:
\begin{equation}\label{Z due}
\int_{\Omega}|Z_t| k \rho \phi(\theta)Z \le C_1 \| \sqrt{\rho}Z_t\|^2_{L^2}+ C_2 \| \nabla Z \|^2_{L^2}.
%\begin{split}
% & \int_{\Omega}|Z_t| k \rho \phi(\theta)Z \le c \| \sqrt{\rho}Z_t\|_{L^2}\| \| \sqrt{\theta}\|_{L^\infty}\|Z\|_{L^4} \\ &  \le  C \| \sqrt{\rho}Z_t\|_{L^2}\| \nabla Z \|_{L^2} \le C_1 \| \sqrt{\rho}Z_t\|^2_{L^2}+ C_2 \| \nabla Z \|^2_{L^2}
%\end{split}
\end{equation}
Summing up \eqref{Z uno} and \eqref{Z due} and by using Poincarè inequality we get
\begin{equation}\label{intZ2}
\dfrac{D}{2}\dfrac{d}{dt}\| \nabla Z\|^2_{L^2}+C_1\|\sqrt{\rho}Z_t\|^2_{L^2} \le C_2 \| \nabla Z\|^2_{L^2},
\end{equation}
  with $C_1,C_2$ positive constants.
By using Gronwall's lemma we have 
\begin{equation}\label{rgZ}
\nabla Z \in L^p (0, T; L^2)  \ \ \ \text{ for any} \ p \ge 1 .
\end{equation}
Moreover, if we integrate \eqref{intZ2} with respect to time we have: $$\dfrac{D}{2}\| \nabla Z(T)\|^2_{L^2}+C_1 \int_{0}^{T} \|\sqrt{\rho}Z_t\|^2_{L^2} \le C_2 \int_{0}^{T}  \| \nabla Z\|^2_{L^2}+\| \nabla Z_0\|^2_{L^2}$$ from which $$\sqrt{\rho}Z_t \in L^2(0, T); L^2),$$ or equivalentely by \eqref{limitrho} 
\begin{equation}\label{rZt}
Z_t \in L^2((0, T); L^2).
\end{equation}
Then, if we take the equation \eqref{Zl} and we integrate over $\Omega$ we have 
\begin{equation}\label{Z3}
\dfrac{d}{dt}\int_{\Omega} \rho Z=-k\int_{\Omega} \phi(\theta)\rho Z \le 0
\end{equation}
hence 
\begin{equation}\label{rZ}
Z\in L^{p}(0, T; L^1) \ \ \ \text{ for any} \ p \ge 1 ,
\end{equation}
\\
Now,  similarly to \cite{Donatelli}  we look for an estimate for $\Delta Z.$  
From equation \eqref{Zl}  it follows:
 $$\int_{\Omega}\Delta Z\le C\int_{\Omega} Z_t+ \nabla Z+ Z.$$ 
 and by considering \eqref{rgZ},\eqref{rZt} and \eqref{rZ} we get 
 $$\Delta Z \in L^2((0, T); L^2),$$ and
 \begin{equation}\label{morl}
 Z \in L^2(0, T); H^2).
 \end{equation}
 
 \subsection{Estimates for $\theta$ and $p$}
We start by deriving some estimates for the temperature $\theta,$ then the estimates for the pressure $p$ follow because of the constitutive law $p=R\rho\theta.$
\\
We apply$\dfrac{\partial{}}{\partial{t}}$ to the equation \eqref{temperaturel},  we multiply by $\theta_t$ and integrate in $\Omega,$
and we have 
\begin{equation}\label{45}
\begin{split}
& \dfrac{1}{2} \dfrac{d}{dt} \int_{\Omega} \rho |\theta_t|^2 dx+ k\int_{\Omega}  | \nabla \theta_t|^2dx  \\ & \le C \int_{\Omega}(| \rho_t | | v | |\nabla \theta| | \theta_t| + \rho |v_t | | \nabla {\theta}| | \theta_t|+ \rho |v| | \nabla {\theta_t}| | \theta_t|  + |p_t| | \nabla v| |\theta_t|  \\ & + \rho |\theta| | \nabla {v_t}| |\theta_t|+ | \nabla v| |\nabla {v_t}| |\theta_t|+| \rho_t| |h| |\theta_t|) dx + \left \langle h_t,\rho \theta_t \right  \rangle =\sum_{j=1}^8 I_j.
\end{split}
\end{equation}
We estimate each term on the right hand side of \eqref{45}.
The estimates of $I_j, j=1,,,,6 $ are similar to \cite{Cho 2006 bis},  for completeness we recall them here.
\begin{equation}
\begin{split}
 I_1 & \le C \| \rho_t\|_{L^3}\|v\|_{L^\infty} \| \nabla \theta \|_{L^2}\| \nabla {\theta_t} \|_{L^2}  \le Cc^6_{2}\| \nabla \theta \|_{L^2}+\dfrac{k}{14}| \nabla {\theta_t} \|_{L^2},
\end{split}
\end{equation}
\begin{equation*}
\begin{split}
 I_2,I_5 & \le C \| \rho \|^{\frac{1}{2}}_{L^\infty}\| v_t \|_{L^6}\| \nabla \theta \|_{L^2}\| \sqrt{\rho}\theta_t\|_{L^3} \\ &  \le C c^3_{0} \| \nabla \theta \|^4|_{L^2}\| \sqrt{\rho}\theta_t\|^2_{L^2}+\dfrac{k}{14} | \nabla {\theta_t} \|_{L^2}+ \| \nabla {v_t}\|^2_{L^2},
\end{split}
\end{equation*}
$$I_3 \le C \| \rho\|^\frac{1}{2}_{L^\infty}\|v\|_{L^\infty}\| \nabla {\theta_t} \|_{L^2}| \sqrt{\rho}\theta_t\|_{L^2} \le Cc^3_{2}| \sqrt{\rho}\theta_t\|^2_{L^2}+\dfrac{k}{14} | \nabla {\theta_t} \|^2_{L^2},$$

\begin{equation*}
\begin{split}
I_4 & \le \|p_t\|_{L^2} \| \nabla v\|_{L^3}\| \theta_t\|_{L^6}
\le Cc^3_{2}\| \sqrt{\rho}\theta_t\|^2_{L^2}+Cc^6_{2}\| \nabla \theta \|^2_{L^2}+\dfrac{k}{14} | \nabla {\theta_t} \|_{L^2},
\end{split}
\end{equation*}

\begin{equation*}
\begin{split}
I_6 \le C \| \nabla v \|^{\frac{1}{2}}_{L^2}\| \nabla {v}\|^{\frac{1}{2}}_{H^1}\| \nabla{v_t} \|_{L^2} \| \nabla {\theta_t} \|_{L^2} \le Cc_1c_2 \| \nabla {v_t} \|^2_{L^2}+\dfrac{k}{14} | \nabla {\theta_t} \|_{L^2},
\end{split}
\end{equation*}
where the constant $c_2$ is chosen as in \eqref{RHO L3}.
Now we estimate $I_7$ and $I_8$ where we have to consider the Arrhenius law
\begin{equation*}
I_7 = \int_{\Omega} | \rho_t |qk | \phi(\theta)| Z | \theta_t|   \le  c\| \rho_t\|_{L^2} \| \sqrt{\theta}\|_{L^4}\| \theta_t\|_{L^4} 
\le \tilde{c} \| \theta \|^{\frac{1}{2}}_{L^2} \| \nabla{\theta_t}\|_{L^2},
%\begin{split}
%I_7 & = \int_{\Omega} | \rho_t |qk | \phi(\theta)| Z | \theta_t| \le \int_{\Omega} | \rho_t| | \sqrt{\theta} | \theta_t| \\ &  \le  c\| \rho_t\|_{L^2} \| \sqrt{\theta}\|_{L^4}\| \theta_t\|_{L^4} \le \tilde{c} \| \theta \|^{\frac{1}{2}}_{L^2} \| \nabla{\theta_t}\|_{L^2},
%\end{split}
\end{equation*}

\begin{equation*}
\begin{split}
I_8= \int_{\Omega} \rho |h_t| |\theta_t|= \int_{\Omega} \rho | \phi_t(\theta)|Z | \theta_t | + c \int_{\Omega} \rho | \phi(\theta) | |Z_t | | \theta_t |= I_{8,1}+ I_{8,2},
\end{split}
\end{equation*}
\begin{equation*}
\begin{split}
 I_{8,1} \le c \int_{\Omega} | \theta_t|^2 e^\frac{-1}{\theta} \bigg[ \frac{1}{2\sqrt{\theta}}+\frac{1}{\sqrt{\theta^3}}\bigg] \le c \| \theta_t\|^2_{L^2},
 \end{split}
 \end{equation*}
 
 \begin{equation*}
 \begin{split}
 I_{8,2} & \le c \int_{\Omega} | \sqrt{\theta}| |Z_t | | \theta_t| \le || \sqrt{\theta}\|_{L^4} \| \sqrt{\rho}Z_t \|_{L^2} \| \theta_t \|_{L^4} \\ & \le \tilde{c} \| \theta \|^{\frac{1}{2}}_{L^2} \| \sqrt{\rho}Z_t \|_{L^2} \| \nabla{\theta_t} \|_{L^2}  \le C_1 \| \theta \|_{L^2} \| \sqrt{\rho}Z_t\|^2_{L^2}+ C_2 \| \nabla{\theta_t}\|^2_{L^2}  \\ & \le \tilde{C_1} \| \sqrt{\rho}Z_t\|^2_{L^2}+ \tilde{C_2}\| \nabla{\theta_t}\|^2_{L^2}.
 \end{split}
 \end{equation*}

By summing up all the estimates for $I_i,$ $ i=1,....,6$ and going back to \eqref{45} we get,
\begin{equation}\label{gronwall per theta}
\begin{split}
& \dfrac{d}{dt} \| \sqrt{\rho}\theta_t\|^2_{L^2}+k \|\nabla \theta_t\|^2 \le Cc^6_{2}(1+\| \nabla \theta \|^4_{L^2})(1+\| \sqrt{\rho}\theta_t\|^2_{L^2}) \\ & +C(c_1c_2\| \nabla {v_t} \|^2_{L^2}+ c^2_0\|Z_t\|^2_{H^{-1}}+c^4_{2}\|Z\|^2_{L^2}).
\end{split}
\end{equation}
\\
By integrating \eqref{gronwall per theta} in $(\tau, t)$ we obtain:
\begin{equation}\label{46}
\begin{split}
& \| \sqrt{\rho}\theta_t (t) \|^2_{L^2}+k\int_{\tau}^{t} \|\nabla \theta_t\|^2_{L^2}ds \le \| \sqrt{\rho}\theta_t (\tau) \|^2_{L^2}+Cc^3_{1}c_2+Cc^4_{2}t \\ & +Cc^6_{2}\int_{\tau}^{t} ( 1+ \| \nabla \theta \|^4_{L^2})(1+\| \sqrt{\rho}\theta_t (t) \|^2_{L^2})ds,
\end{split}
\end{equation}
for $ 0< \tau <t\le \min(T_*,T_1).$
Now,  in order to estimate $\limsup \limits_{ \tau \rightarrow{0}} \| \sqrt{\rho}\theta_t (\tau) \|^2_{L^2}$ we observe that from \eqref{temperaturel} we have: 
$$ \int_{}^{} \rho \theta_t dx \le C \int_{}^{} (\rho |h|^2+ \rho |v|^2| \nabla \theta |^2 + \rho |\theta|^2 | \nabla v|^2 + \rho^{-1} |k \Delta \theta +Q( \nabla v)|^2)dx$$
and so,  using the condition \eqref{c0} we get:
\begin{equation*}
\begin{split}
\limsup \limits_{ \tau \rightarrow{0}} \| \sqrt{\rho}\theta_t (\tau) \|^2_{L^2} & \le C( \| \rho_0\|_{L^\infty}\|h(0)\|^2_{L^2} \\ & +\| \rho_0\|_{L^\infty}\| \nabla {u_0}\|^2_{H^1}\| \nabla {u_0}\|^2_{L^2}+\|g_1\|^2_{L^2})\le Cc^5_{0}
\end{split}
\end{equation*}
and sending $\tau \rightarrow{0}$ in \eqref{46},  we have that:
\begin{equation}\label{3.30}
\begin{split}
& \| \sqrt{\rho}\theta_t (t) \|^2_{L^2}+k\int_{0}^{t} \|\nabla \theta_t\|^2_{L^2}ds \\ & \le Cc^4_{1}c_2+Cc^6_{2}\int_{0}^{t} ( 1+ \| \nabla \theta \|^4_{L^2})(1+\| \sqrt{\rho}\theta_t (t) \|^2_{L^2})ds,
\end{split}
\end{equation}
for $0\le t \le \min(T_*,T_2), $ where $T_2=c^{-4}_2< T_1.$
Moreover we observe that 
\begin{equation}\label{numero}
 \| \nabla {\theta(t)}\|^2_{L^2} \le C\| \nabla {\theta_0} \|^2_{L^2}+C \int_{0}^{t} \| \nabla {\theta_t} \|^2_{L^2}ds ,
 \end{equation}
  for $0 \le t \le \min(T_*, T_2).$
\\
By summing up \eqref{3.30} and \eqref{numero} it follows
\begin{equation}\label{48}
\begin{split}
& \| \nabla {\theta(t)}\|^2_{L^2}+\| \sqrt{\rho}\theta_t\|^2_{L^2}+\int_{0}^{t} \| \nabla {\theta_t} \|^2_{L^2}ds \\ &  \le Cc^4_{1}c_2+Cc^6_{2} \int_{0}^{t} ( 1+ \| \nabla \theta \|^4_{L^2})(1+\| \sqrt{\rho}\theta_t (t) \|^2_{L^2})ds,
 \end{split}
\end{equation}
for $ 0 \le t\le \min(T_*,T_2).$ Now we define $\Gamma(t)$ as $$\Gamma(t)=1+\| \theta (t)\|^2_{H^1_{0}}+\| \sqrt{\rho}\theta_t (t) \|^2_{L^2}).$$
Then we can rewrite \eqref{48} in the following form,
$$\Gamma(t) \le Cc^4_{1}c_2+Cc^6_{2}\int_{0}^{t} \Gamma(s)^3 ds,$$
for $ 0 \le t\le \min(T_*,T_2).$
It is an integral inequality that can be solved.  Indeed we have,
\begin{equation}
\begin{cases}
\Gamma'(t) \le Cc^6_{2}\Gamma(s)^3 \\
\Gamma(0) \le Cc^4_{1}c_2 
\end{cases}
\end{equation} so that by separation of variables we get,
 $$\Gamma(t) \le Cc^4_{1}c_2(1-Cc^{16}_{2}t)^{-\frac{1}{2}},$$ for small $t\ge 0.$
By choosing $T_3=(2Cc^{16}_{2})^{-1} $ with $C>1 $ we have the following inequality,
 \begin{equation}\label{49}
 \| \theta(t) \|^2_{H^1_{0}}+\| \sqrt{\rho}\theta_t\|^2_{L^2}+\int_{0}^{t}(|\theta_t(s)|^2_{H^1_{0}})ds \le Cc^4_{1}c_{2}
 \end{equation}
for $0 \le t \le \min(T_*, T_3).$
\\
Regarding the integrals involving the heat source $h$
we recall that  $$h=qk\phi(\theta)Z$$
hence,  using the regularity \eqref{regularity for Z L2} together with \eqref{49} we end up with the regularity assigned in \cite{Cho 2006 bis} that is 
$$ h\in C([0, T]; L^2)\cap L^2(0,T; L^q)\quad \text{and} \quad h_t \in L^2(0, T; H^{-1}).$$
This regularity for the heat source $h$ leads to the following further estimates for the temperature $\theta$
\begin{equation*}
\begin{split}
\| \nabla \theta \|_{H^1}
\le Cc^2_{2}( 1+ \| \sqrt{\rho}\theta_t \|_{L^2}+ \| \nabla \theta \|_{L^2}),
\end{split}
\end{equation*}
\begin{equation}\label{3.35}
\| \theta(t) \|_{H^2} \le Cc^5_{2}, \quad \int_{0}^{t} \| \theta(s)\|^2_{W^{2,q}}ds \le Cc^7_{2},
\end{equation}
for $0 \le t \le \min(T_*, T_3).$ We refer to Section $2.2$ in \cite{Cho 2006 bis} for further details.
Finally,  regarding the pressure $p$ we recall the constitutive law $p=R\rho \theta$ from which we obtain
\begin{equation}\label{REG PRESS}
\| \nabla {p(t)} \|_{L^2} \| \le Cc^3_{1}c^{\frac{1}{2}}_{2}, \quad \| \nabla {p(t)} \|_{L^q} \| \le Cc^6_{2}. \quad \| p_t(t)\|_{L^2} \le Cc^5_{2}
\end{equation}
for $0 \le t \le \min(T_*, T_3).$
 \subsection{Estimates for u}
 Regarding the regularity estimates for the velocity $u$ recall that the external force $f$ is of the gravitational type i.e.
 $$ f=-\nabla\Phi,$$ 
 where the potential $\Phi$ is given by the following Poisson equation $$\Delta\Phi=4\pi G \rho.$$
 Now,  because of \eqref{limitrho} we get from the elliptic theory that our external force $f$ has at least  the regularity assumed in  \cite{Cho 2006 bis} for the external force that is 
\begin{equation}\label{reg f}
f \in C([0, T]; L^2)\cap L^2(0,T; L^q), \quad  f_t \in L^2(0, T; H^{-1}),
\end{equation}
 for a detailed discussion see \cite{Meyers}.
 \\
With the same lines of arguments as in \cite{Cho 2006 bis}, Section 2.3,  we obtain the following estimate for the velocity $u$
 \begin{equation}\label{reg U}
 \begin{split}
 \| u(t) \|^2_{H^1_{0}}+c^7_{0}c^{-6}_1c^{-\frac{1}{2}} \| u(t)\|_{H^2}+\| \sqrt{\rho}u_t\|_{L^2} \\ +\int_{0}^{t}(|u_t(t)|^2_{H^1_{0}}+|u(t)|^2_{W^{2,q}})ds \le Cc^7_{0}
 \end{split}
 \end{equation}
for $0 \le t \le \min(T_*, T_3),$
 \subsection{High order estimates for Z}
We conclude the estimates for the linearized system by getting an higher order regularity estimate for the mass fraction $Z$.  We recover this estimate at the end of the section because it strongly depends on the regularity of the time derivative of the temperature. 
 \\
We apply the time derivative on the equation \eqref{Zl},  we multiply it by $Z_t$ and we integrate over $\Omega.$
By using the continuity equation and the boundary condition we have 
%\begin{equation*}
% \begin{split}
%\int_{\Omega}\rho_t Z^2_t+ \rho	Z_{tt}Z_t+\rho_t v\cdot \nabla ZZ_t+ \rho v_t \cdot \nabla Z Z_t+\rho v \nabla {Z_t}Z_t= \\ -\int_{\Omega}k\phi_t(\theta)\rho Z Z_t-K\phi(\theta)\rho_t Z Z_t  -\int_{\Omega} k\phi(\theta)\rho Z^2_t+\operatorname{div}(D\nabla {Z_t})Z_t.
%\end{split}
%\end{equation*}
 %by using the continuity equation and the boundary condition we can write 
\begin{equation}\label{ineqZ5}
\begin{split}
& \dfrac{1}{2}\dfrac{d}{dt}\int_{\Omega}\ \rho Z^2_t+ D\int_{\Omega} |\nabla Z_t|^2 \\ & \le \int_{\Omega} |\rho_t| |v| \cdot |\nabla Z| |Z_t|+\int_{\Omega} \rho |v_t|  \cdot |\nabla Z| |Z_t|+ k\int_{\Omega} |\phi_t(\theta)| \rho Z |Z_t| \\ & +k\int_{\Omega}|\phi(\theta)||\rho_t| Z |Z_t|+k\int_{\Omega}|\phi(\theta)|\rho Z^2_t=\sum_{j=1}^5 I_j.
\end{split}
\end{equation}
As before we analyse each one of the terms $I_j,$ $ j=1,...,5.$ Hence we get 
\begin{equation}\label{I1hz}
\begin{split}
I_1 & \le \|v\|_{L^\infty} \int_{\Omega} |\rho_t| \cdot |\nabla Z| |Z_t|\   \le c\| \rho_t\|_{L^3} \| \nabla Z\|_{L^2} \| \nabla Z_t\|_{L^2} \\ & \le \tilde{c} {\| \nabla Z\|}^2_{L^2}+ {\| \nabla Z_t\|}^2_{L^2},
\end{split}
\end{equation}

\begin{equation}\label{I2hz}
\begin{split}
I_2 & \le c \| \nabla v_t\|_{L^2} \|\nabla Z\|_{L^2} \| \nabla Z_t\|_{L^2} \\ & \le \tilde{c} ( {\|\nabla Z\|}^2_{L^2} + {\| \nabla Z_t\|}^2_{L^2}),
\end{split}
\end{equation}
\begin{equation}\label{I3hz}
\begin{split}
I_3 & \le c \int_{\Omega} | \phi_t(\theta)| \rho Z |Z_t|  \le c \int_{\Omega}|\theta_t| e^\frac{-1}{\theta} \bigg[\dfrac{1}{2\sqrt{\theta}}+\dfrac{1}{\sqrt{\theta}^3} \bigg]\rho Z |Z_t|\\ & \le c\| \sqrt{\rho}Z_t\|_{L^2}c\| \sqrt{\rho}\theta_t\|_{L^2} \le C_1| |\sqrt{\rho}Z_t\|^2_{L^2}+ C_2\| \sqrt{\rho}\theta_t\|^2_{L^2},  
\end{split} 
\end{equation}

\begin{equation}\label{I4hz}
\begin{split}
I_4 & = k \int_{}^{} | \phi(\theta)| | \rho_t| |Z| |Z_t| \le k \| \rho_t\|_{L^2} \| \sqrt{\theta}\|_{L^4}\| Z_t\|_{L^4} \\ & \le  c( \| \theta \|^{\frac{1}{2}}_{L^2} \| \nabla {Z_t}\|_{L^2})  \le \tilde{c} \| \nabla {Z_t}\|_{L^2},
\end{split}
\end{equation}

\begin{equation}\label{I5hz}
\begin{split}
I_5 & = \int_{\Omega} | \phi(\theta)| \rho Z^2_t=\int_{\Omega}\sqrt{\theta}e^{-\frac{1}{\theta}}\rho Z^2_t  \le \| \sqrt{\rho}Z_t\| \|Z_t\|_{L^4}\| \sqrt{\theta}\|_{L^4} \\ & \le \| \sqrt{\rho}Z_t\| \| \nabla {Z_t} \|_{L^2}\| \theta \|^{\frac{1}{2}}_{L^2}  \le C_1 \| \sqrt{\rho}Z_t \|^2_{L^2}+ C_2 \| \nabla {Z_t}\|^2_{L^2}.
\end{split}
\end{equation}
Summing up \eqref{I1hz}-\eqref{I5hz} and by choosing properly the constants of the Young inequality we get the following estimates
\begin{equation}\label{GronwdtZ}
\begin{split}
& \dfrac{1}{2} \dfrac{d}{dt} \| \sqrt{\rho}Z_t\|^2_{L^2} + C_1 \| \nabla Z_t\|^2_{L^2}  \\ & \le  \tilde{c}\| \nabla Z \|^2_{L^2} \| + \tilde{c_1} { \| \sqrt{\rho} \theta_{t} \|}^2_{L^2} + C_2 \| \sqrt{\rho} Z_t \|^2_{L^2}.
\end{split}
\end{equation}
Now,  by applying Gronwall inequality we have:
\begin{equation}\label{RZt}
\sqrt{\rho}Z_t \in L^\infty ((0, T; L^2).
\end{equation}
Moreover,  integrating \eqref{GronwdtZ} in $(0, T)$ we get:

\begin{equation}
\begin{split}
\| \sqrt{\rho}Z_t (T) \|^2_{L^2} + \int_{0}^{T} C_1 \| \nabla Z_t\|^2_{L^2} & \le  \tilde{c} \int_{0}^{T}  \| \nabla Z \|^2_{L^2} \| + \tilde{c_1} \int_{0}^{T}  { \| \sqrt{\rho} \theta_{t} \|}^2_{L^2} \\ & +C_2  \int_{0}^{T} \| \sqrt{\rho} Z_t \|^2_{L^2} + \| \sqrt{\rho}Z_t (0) \|^2_{L^2},
\end{split}
\end{equation}
and by using \eqref{RZt} we end up with  
\begin{equation}\label{H10}
Z_t \in L^\infty (0, T; H^1_0).
\end{equation}
Now,  combining \eqref{morl} with \eqref{H10} we obtain 
\begin{equation}\label{Z H1}
Z \in H^1(0, T; H^1).
\end{equation}
which by Morrey inequality guarantees the continuity in time for $Z.$
\subsection{The limit $\delta \rightarrow 0$}
We can summarise all the estimates of the previous section as follows
\begin{equation}\label{59}
\begin{split}
\sup \limits_{0\le t\le T_*} (| \rho (t)|_{H^1 \cap W^{1, q}} +| \rho_t(t)|_{L^2 \cap L^q}+ |\theta(t)|_{H^1_{0} \cap H^2}) \le Cc^7_{2}, \\  ess \sup \limits_{0\le t\le T_*} (\sqrt{\rho}\theta_t, \sqrt{\rho}u_t,  \sqrt{\rho}Z_t|_{L^2} + \int_{0}^{T_*} (|\theta_t(t)|^2_{H^1_{0}}+|\theta(t)|^2_{W^{2,q}})dt \le Cc^7_{2}, \\
\sup \limits_{0\le t\le T_*} (|u(t)|_{H^1_{0}}+ \beta^{-1} |u(t)|_{H^2})+ \int_{0}^{T_*}(|u_t(t)|^2_{H^1_{0}}+|u(t)|^2_{W^{2,q}})dt \le c_1, \\
\sup \limits_{0\le t\le T_*} (|Z(t)|_{H^1_{0}}+ |Z(t)|_{H^2})+ \int_{0}^{T_*}(|Z_t(t)|^2_{H^1_{0}}+|Z(t)|^2_{W^{2,q}})dt \le c_3.
\end{split}
\end{equation}
We point out that all the constants are independent on the lower bounds of the initial density $\rho_0.$
In the next lemma,  which is a generalisation of Lemma 9 in \cite{Cho 2006 bis} we get the existence of solutions to the linearized problem without assuming a positive lower bound for the initial density.

\begin{lemma}
\label{lemma9}
Assume in addition to \eqref{reglin} that the initial data $(\rho_0,\theta_0,u_0, Z_0)$ satisfy the compatibility condition \eqref{compat cond} and the vector field $v$ has the regularity \eqref{regularity v}
and 

$$c_0= 2+ \|\rho_0\|_{H^1 \cap W^{1, q}} + \|(\theta_0, u_0, Z_0)\|_{H^1_{0} \cap H^2}+ \|(g_1, g_2)\|^2_{L^2}.$$
Then there exists a unique strong solution $(\rho, \theta, u, Z)$ to the linearized problem \eqref{densityl}-\eqref{bcl} in $[0, T_*]$ satisfying the estimates \eqref{59} as well as the regularity $$ \rho \in L^{\infty}([0,T_*];H^{1} \cap W^{1,q}),  \quad \rho_t\in L^{\infty}([0, T_*]; L^2 \cap L^q), $$
$$(\theta,  u,  Z) \in L^\infty([0, T_*]; H^1_{0} \cap H^2) \cap L^2(0, T_*; W^{2,q}),$$
\begin{equation}
(\theta_t,  u_t,  Z_t)\in L^2(0, T_*; H_{0}^1), \quad (\sqrt{\rho}\theta_t,   \sqrt{\rho}u_t, \sqrt{\rho}Z_t) \in L^\infty (0, T_*; L^2).
\end{equation}

\end{lemma}

\begin{proof}

Define $\rho^\delta _0= \rho_0 +\delta $ for each $\delta \in (0,1).$ Then from the compatibility condition \eqref{compat cond} we have $$-k\Delta \theta_0-Q(\nabla u_0)=(\rho^\delta_{0})^\frac{1}{2}g^{\delta}_1 \quad and \quad Lu_0+ R\nabla {\rho^\delta_0} \theta_0=(\rho^{\delta}_{0})^\frac{1}{2}g^{\delta}_2$$

where $$g^{\delta}_1= \left( \dfrac{\rho^{\delta}_0}{\rho_0}\right)^{\frac{1}{2}}g_1, \quad g^{\delta}_2= \left(\dfrac {\rho_0} {\rho^{\delta}_0}\right)^{\frac{1}{2}} g_2+ R \dfrac{\delta}{\left(\rho^{\delta}_{0}\right)^\frac{1}{2}}\nabla \theta_0.$$
Moreover we observe that for all small $\delta>0,$
$$c_0= 1+ \delta + |\rho^{\delta}_0- \delta |_{H^1 \cap W^{1, q}} + |(\theta_0, u_0, Z_0)|_{H^1_{0} \cap H^2}+ |(g^{\delta}_1, g^{\delta}_2)|^2_{L^2}.$$
So we have that for the corresponding initial data $(\rho^{\delta}_0, \theta_0, u_0, Z_0)$ with small $\delta >0$  there exists a unique strong solution $(\rho^{\delta}, \theta^{\delta}, u^{\delta}, Z^{\delta})$ of the linearized system \eqref{densityl}-\eqref{Zl} satisfying the local estimates \eqref{59}. 
Now,  as $\delta \rightarrow 0$ we have that $(\rho^{\delta}, \theta^{\delta}, u^{\delta}, Z^{\delta})$ converges in a weak or weak* sense to 
$(\rho, \theta,  u, Z)$.  Because of the independence of the estimates \eqref{59} on $\delta$ and due to the lower semicontinuity of the norms,  we have that $(\rho, \theta,  u, Z)$  satisfies \eqref{59}.  Moreover $(\rho, \theta,  u, Z)$ is a solution to the linearized problem with general initial data $(\rho_0, \theta_0,  u_0, Z_0)$ with nonnegative initial density and it satisfies the regularity:
$$ \rho \in L^{\infty}([0,T_*];H^{1} \cap W^{1,q}),  \quad \rho_t\in L^\infty([0, T_*]; L^2 \cap L^q), $$
$$(\theta,  u,  Z) \in L^\infty([0, T_*]; H^1_{0} \cap H^2) \cap L^2(0, T_*; W^{2,q}),$$
\begin{equation}\label{3.59}
(\theta_t,  u_t,  Z_t)\in L^2(0, T_*; H_{0}^1), \quad (\sqrt{\rho}\theta_t,   \sqrt{\rho}u_t, \sqrt{\rho}Z_t) \in L^\infty (0, T_*; L^2).
\end{equation}
\subsection{Uniqueness}
The last step is to prove the uniqueness of the solution.
We assume by contradiction that there exists two solutions $(\rho_1, \theta_1,  u_1,  Z_1) $ and $(\rho_2, \theta_2,  u_2,  Z_2) ,$ and we define: $$ \overline{\rho}=\rho_1-\rho_2,$$ $$\overline{\theta}=\theta_1-\theta_2,$$
$$ \overline{u}=u_1-u_2,$$ $$ \overline{Z}=Z_1-Z_2.$$
Since $\overline{\rho} \in L^{\infty}(0, T_*; L^q)$ is a solution to the linear transport equation $$\overline{\rho}_t+\operatorname{div}(\overline{\rho}v)=0 $$ then,  by the same argument of the proof of Lemma \ref{lemma6},  it follows that necessarily $\overline{\rho}=0$ i.e.  $\rho_1=\rho_2$ in $[0, T_*] \times {\Omega}.$ \\
Now we take the equation: 
\begin{equation}
\rho_1 \overline{\theta}_t +\rho_1 v \cdot \nabla \overline{\theta}+ R \rho_1 \overline{\theta}\operatorname{div}v=0
\end{equation} 
and multiply by $\overline{\theta}$ and integrate in $(0,t) \times \Omega,$ recalling that $({\rho_1})_t+\operatorname{div}(\rho_1 v)=0$ and $\overline{\theta}(0)=0$ in $\Omega$ we deduce that:
\begin{equation}\label{uniqueness}
\dfrac{1}{2} \int_{\Omega} \rho_1| \overline{\theta}|^2(t)dx+ k \int_{0}^{t} \int_{\Omega} | \nabla {\overline{\theta}}|^2 dxds=-R  \int_{0}^{t}\int_{\Omega} \rho_1|\overline{\theta}|^2 | \operatorname{div}v|dxds
\end{equation}
and applying Gronwall's inequality we have $\overline{\theta}=0$ in $(0,t) \times \Omega.$ \\
Note that,  since we are working on a bounded domain,  this estimate is rigorous and we don't need to use the expansion technique by means of the cut off functions $\psi \in C^\infty _c (B_1)$ as in \cite{Cho 2006 bis} that were used in order to justify the boundness of the left hand side in \eqref{uniqueness}. 
Following the same procedure,  one can also prove that $\overline{u}=0,\quad \overline{Z}=0$ in $[0, T_*] \times {\Omega}.$\\
Concerning the continuity in time of the solution $(\rho, \theta, u, Z)$ we recall for the density that,  $\rho \in L^{\infty}(0, T_*; L^q)$ hence the continuity in time follows from Lemma \ref{lemma6}.
Regarding $(\theta,  u, Z)$ we already have that $$(\theta,  u, Z) \in C([0, T_*]; H^1_0)\cap C([0, T_*]; H^2-weak)$$ and from equations \eqref{momentuml},  \eqref{temperaturel} it follows that $$((\rho \theta _t)_t,(\rho u _t)_t,(\rho Z_t)_t) \in L^2 (0, T_*; H^{-1}).$$
Combining this two estimates with $$(\rho \theta_t, \rho u_t,  \rho Z_t) \in L^2( 0, T_*; H^1_0)$$ we end up with
\begin{equation}
(\rho \theta_t,  \rho u_t,  \rho Z_t) \in C([0. T_*]; L^2).  
\end{equation}
hence  we deduce that
 \begin{equation}\label{**}
(\theta, u, Z) \in C([0,T_*]; H^2).
\end{equation}
\end{proof}

\section{ Proof of the Theorem 1.1}
This section is devoted to the proof of Theorem \ref{Teo2.1}. 
The proof is performed by using an iteration argument.
We properly define an iterative scheme for \eqref{pb1} and we prove the existence of an approximating solution a fixed step $k$ of the inductive process.  We estimate the difference between the solution at two different steps $k,k+1$ and finally we show the convergence of the approximating solution to the solution of the nonlinear problem \eqref{pb1}.
We denote 
\begin{equation}
c_0 = 2+ \| \rho_0\|_{H^1 \cap W^{1, q}}+ \| (\theta_0,  u_0, Z_0)\|_{H^1_{0} \cap H^2}+ \|( g_1, g_2)\|^2_{L^2},
\end{equation}
and,  as in the previous section,  we choose the positive constants $c_1,\beta, c_2$ and $T_{**}$ depending  on $c_0.$ \\
Now,  let $u^0 \in C([0, \infty ); H^1_0 \cap H^2) \cap L^2(0, \infty; H^2)$ be the solution of the heat equation 
\begin{equation}\label{heat}
\begin{cases}
\omega_t-\Delta \omega=0\\
\omega(0)=u_0 \, ,
\end{cases}
\end{equation}
If we take a small time $T_1 \in (0, T_{**}]$ we have $$ \sup \limits_{0 \le t \le T_*} {(|u^0(t)|_{H^1_0}+\beta^{-1}u^0(t)|_H^2)} + \int_{0}^{T_*} (|u^0_t(t)|^2_{H^1_0} + |u^0t)|^2_{W^{2,q}})dt \le c_1$$ 
By applying Lemma \ref{lemma9} there exists a unique strong solution $(\rho_1, \theta_1, u_1, Z_1)$ of the linearized system \eqref{densityl}-\eqref{bcl} with $v=u^0,$ which satisfies the regularity estimate \eqref{59} with $T_*$ replaced by $T_1.$\\
We define the following inductive scheme.  We take $u^0$ as in \eqref{heat}, we assume  that $u^{k−1}$ is defined for $k> 1, $ and we denote by $(\rho^k,\theta^k, u^k, Z^k)$  the unique strong solution to the following initial value problem

\begin{equation}\label{densitylk}
\rho^k_t+ \operatorname{div}(\rho^k u^{k-1})=0 
\end{equation}
\begin{equation}\label{momentumlk}
(\rho^k u^k)_{t}+ \operatorname{div} (\rho^k u^{k-1} \otimes u^k)+ Lu^k+ \nabla p^k= \rho^k f^k
\end{equation}
\begin{equation}\label{temperaturelk}
c_v((\rho^k \theta^k)_t+\operatorname{div}(\rho^k \theta^k u^{k-1}))-k\Delta \theta^k+ p^k\operatorname{div}u^{k-1}=Q(\nabla u^{k-1})+ \rho^k h^k 
\end{equation}
\begin{equation}\label{Zlk}
(\rho^k Z^k)_{t}+\operatorname{div}(\rho^k u^{k-1}Z^k))=-k \phi (\theta^k) \rho^k Z^k +\operatorname{div}(D \nabla Z^k),
\end{equation}
\begin{equation}\label{poisson lk}
\Delta \Phi^k=4\pi G\rho^k
\end{equation}
\begin{equation}\label{ic1lk}
(\rho,^k \theta^k, u^k, Z^k)_{|{t=0}}=(\rho_0, \theta_0, u_0, Z_0) \quad in \quad \Omega
\end{equation}
\begin{equation}\label{bclk}
(\theta^k, u^k,Z^k)=(0,0,0) \quad on\quad (0, T)\times \partial\Omega
\end{equation}
\\
where $t\geq 0$, $x\in \Omega\subset\mathbb{R}^{3}$  $$f^k=-\nabla \Phi^k,\quad \quad h^k=qk\phi(\theta^k)Z^k, \quad p^k=R\rho^k \theta^k, $$
$$Q(\nabla u^k)=\frac{\mu}{2}|\nabla u^k+\nabla {u^k}^T|^2+\lambda (\operatorname{div}u^k)^2, \quad  Lu^k=-\mu \Delta u^k-(\lambda+\mu)\nabla\operatorname{div}u^k.$$ 
Assuming by induction that $u^{k−1}$ is defined for $k> 1, $ and has the regularity \eqref{3.59} and \eqref{**} by applying Lemma \ref{lemma9} it follows that $(\rho^k, \theta^k,  u^k, Z^k)$ is a strong solution of the problem \eqref{densitylk}-\eqref{bclk} with the regularity 
\begin{equation}\label{65}
\begin{split}
& \sup \limits_{0\le t\le T_1} (| \rho^k (t)|_{H^1 \cap W^{1, q}} +| \rho^k_t(t)|_{L^2 \cap L^q}) + \sup\limits_{0\le t\le T_1}  |( \theta^k,  u^k,  Z^k)(t)|_{H^1_{0} \cap H^2}) \\ & + ess \sup \limits_{0\le t\le T_1} (\sqrt{\rho^k}\theta^k_t, \sqrt{\rho^k}u^k_t,  \sqrt{\rho^k}Z^k_t|)_{L^2} + \int_{0}^{T_1} |( \theta^k_t,  u^k_t, Z^k_t)(t)|^2_{H^1_{0}} \\ & + |( \theta^k_t,  u^k_t, Z^k_t)(t)|^2_{W^{2,q}}dt \le \tilde{C}. 
\end{split}
\end{equation}
indeed for each fixed $k$ we have that the system \eqref{densitylk}-\eqref{bclk} is a linearized system with $v$ repleaced by $u^{k-1}.$
Our goal is to show that the sequence $(\rho^k, \theta^k,  u^k, Z^k)$ converges to a solution of the nonlinear problem \eqref{pb1} in a strong sense.  Now we define,
$$\overline{\rho}^{k+1} =\rho^{k+1} −\rho^k,  \quad  \overline{\theta} ^{k+1}= \theta ^{k+1} −\theta ^k, $$ $$\overline{u}^{k+1}=u^{k+1} −u^k, \quad  \overline{Z}^{k+1}=Z^{k+1}-Z^k .$$
From the equations \eqref{densitylk}-\eqref{bclk} we get the following equations for the differences
\begin{equation}\label{DIFFERENCE DENSITY}
\overline{\rho}^{k+1}_t + \operatorname{div}(\overline{\rho}^{k+1}u^k)+\operatorname{div}(\rho^k \overline{u}^k)= 0,
\end{equation}
\begin{equation}\label{DIFFERENCE TEMP}
\begin{split}
& \rho^{k+1} \overline{\theta}^{k+1}_{t} +\rho^{k+1}u^k \cdot \nabla \overline{\theta}^{k+1}-k\Delta \overline{\theta}^{k+1}=Q(\nabla u^k)-Q(\nabla u^{k-1}) \\ & -\overline{\rho}^{k+1}\theta^k_t  + \overline{\rho}^{k+1}[qk\phi(\theta^k)Z^k-u^{k-1} \cdot \nabla \theta^k- R\theta^k\operatorname{div}u^{k-1}] \\ & -\rho^{k+1}[-qk\phi(\theta^{k+1})\overline{Z}^{k+1}-qk\phi(\theta^{k+1})Z^k+qk\phi(\theta^k)Z^k \\ & +\overline{u}^k \cdot \nabla \theta^k +R \overline{\theta}^{k+1} \operatorname{div}\overline{u}^k+R\theta^k\operatorname{div}\overline{u}^k],
\end{split}
\end{equation}
\begin{equation}\label{DIFFERENCE MOMENT}
\begin{split}
 \rho^{k+1} \overline{u}^{k+1}_{t} &+\rho^{k+1}u^k \cdot \nabla \overline{u}^{k+1}+L \overline{u}^{k+1}\\
 &= \overline{\rho}^{k+1}(f^{k+1}-u^{k}_t-u^{k-1} \cdot \nabla u^k)\\&-\rho^{k+1}\overline{u}^k \cdot \nabla u^k-R\nabla (\rho^{k+1}\overline{\theta}^{k+1}-\overline{\rho}^{k+1}\theta^k)+\rho^k\overline{f}^{k+1},
 \end{split}
 \end{equation}
\begin{equation}\label{DIFFERENCE Z}
\begin{split}
\rho^{k+1} \overline{Z}^{k+1}_{t} +\rho^{k+1}u^k \cdot \nabla \overline{Z}^{k+1}-D\Delta  \overline{Z}^{k+1}=-\overline{\rho}^{k+1}Z^k_t\\+\overline{\rho}^{k+1}[-u^{k-1} \cdot \nabla Z^k-k\phi(\theta^k)Z^k]\\-\rho^{k+1}[\overline{u}^k \cdot \nabla Z^k+ k\phi(\theta^{k+1})\overline{Z}^{k+1} + k\phi(\theta^{k+1})Z^k-k\phi(\theta^k)Z^k].
\end{split}
\end{equation}
We point out that differently from \cite{cho2004} and \cite{Cho 2006 bis},  in our case the external force $f$ and the heat source $h$ depends on the solution,  so the iteration scheme leads to more complicated equation for the differences to deal with.
We perform some a priori estimates for the new unknows $\overline{\rho}^{k+1}$, $\overline{\theta}^{k+1}$,  $\overline{u}^{k+1}$, $\overline{Z}^{k+1}$.
The idea is to deal with the $k+1$ terms by using the regularity of the approximating solution at the step $k$ which is given by the inductive hypothesis and then to construct a Gronwall lemma type structure.
We multiply \eqref{DIFFERENCE Z} by $\overline{Z}^{k+1},$  integrate in $\Omega$ and we use Lagrange Mean Value Theorem in order to deal with the unknown terms in $k+1$ in the right hand side of \eqref{DIFFERENCE Z}.  Hence we obtain the following inequality

\begin{equation}\label{Dis ZK}
\begin{split}
& \dfrac{1}{2}\dfrac{d}{dt} \int_{}^{} \rho^{k+1} {| \overline{Z}^{k+1}|}^2+ D \int_{}^{} {| \nabla {\overline{Z}^{k+1}|}}^2  \le \int_{}^{}| \overline{\rho}^{k+1}||Z^k_t| |\overline{Z}^{k+1}| \\ & +\int_{}^{} |\overline{\rho}^{k+1}|[|u^{k-1}|| \nabla Z^k|+k\phi(\theta^k)Z^k]|\overline{Z}^{k+1}| +\int_{}^{} \rho^{k+1}[|\overline{u}^k| | \nabla Z^k| \\ & +kc|\overline{Z}^{k+1}||\overline{\theta}^{k+1}|+ k\phi(\theta^{k})|\overline{Z}^{k+1}| +kcZ^k|\overline{\theta}^{k+1}|]|\overline{Z}^{k+1}| \\ &  =\int_{}^{} \overline{\rho}^{k+1}|Z^k_t| |\overline{Z}^{k+1}| +\sum_{j=1}^6 I_j.
\end{split}
\end{equation}
Now, by means of H\"{o}lder,  Young and Sobolev inequalities,  we estimate all the terms $I_j$ $j=1,,,,6$ in terms of the left hand side of \eqref{Dis ZK},  hence we have
\begin{equation*}
\begin{split}
I_1 & = \int_{}^{}| \overline{\rho}^{k+1}||u^{k-1}|| \nabla Z^k||\overline{Z}^{k+1}| \le \| u^{k-1}\|_{L^\infty}||\overline{Z}^{k+1}||_{L^\infty}  \| \overline{\rho}^{k+1} \|_{L^2} \| \nabla Z^k \|_{L^2} \\ & \le  C_1\| \overline{\rho}^{k+1} \|^2_{L^2}+ C_2\| \nabla Z^k \|^2_{L^2},
\end{split}
\end{equation*}
\begin{equation*}
\begin{split}
I_2 & =\int_{}^{} | \overline{\rho}^{k+1}k\sqrt{\theta^k}e^{-\frac{1}{\theta^k}}Z^k|\overline{Z}^{k+1} \le  c \| \overline{\rho}^{k+1} \|_{L^2} \| \overline{Z}^{k+1}\|_{L^4}\|\sqrt{\theta^k}\|_{L^4} \\ & \le c\| \overline{\rho}^{k+1} \|_{L^2}\| \overline{ \nabla Z}^{k+1} \|_{L^2}\|\theta^k\|^{\frac{1}{2}}_{L^2} \le C_1 {\| \overline{\rho}^{k+1} \|}^2_{L^2}+C_2{\|  { \nabla \overline{ Z}^{k+1}} \|}^2_{L^2},
\end{split}
\end{equation*}
\begin{equation*}
\begin{split}
I_3 & =\int_{}^{} \rho^{k+1}\overline{Z}^{k+1}|\overline{u}|^k| | \nabla {Z^k}|  \le  \| \sqrt{\rho^{k+1}}\|_{L^\infty} \| \sqrt{\rho^{k+1}}\overline{Z}^{k+1}\|_{L^2}\|\overline{u}^k\|_{L^4}\| \nabla {Z^k}\|_{L^4}\\ & \le c\| \sqrt{\rho^{k+1}}\overline{Z}^{k+1}\|_{L^2}\| \nabla {\overline{u}^k}\|_{L^2}\| \nabla {Z^k}\|_{L^4} \\ & \le  c { \| \nabla {\overline{u}^k} \|}^2_{L^2}+\tilde{c}{( \| \sqrt{\rho^{k+1}}\overline{Z}^{k+1}\|}^2_{L^2}{\| \nabla {Z^k}\|}^2_{L^4}),
\end{split}
\end{equation*}
\begin{equation*}
\begin{split}
I_4 & = \int_{}^{} \rho^{k+1}{\overline{Z}^{k+1}}^2 kc |\overline{\theta}^{k+1}|  \le C_1{\| \sqrt{\rho^{k+1}}\overline{Z}^{k+1}\|}^2_{L^2}+C_2{\| \sqrt{\rho^{k+1}}\overline{\theta}^{k+1}\|}^2_{L^2},
\end{split}
\end{equation*}
\begin{equation*}
\begin{split}
I_5 & = \int_{}^{} \rho^{k+1} |\overline{Z}^{k+1}|^2 k\phi(\theta^{k}) \le \| \sqrt{\rho^{k+1}}\|_{L^\infty} {\| \sqrt{\rho^{k+1}} \overline{Z}^{k+1}\|}_{L^2}\| \theta^k\|^{\frac{1}{2}}_{L^2} \| \nabla{\overline{Z}^	{k+1}} \|_{L^2} \\ & \le C_1{\| \sqrt{\rho^{k+1}} \overline{Z}^{k+1}\|}^2_{L^2}+ C_2 \| \nabla{\overline{Z}^{k+1} \|}^2_{L^2},
\end{split}
\end{equation*}
\begin{equation*}
\begin{split}
I_6 & = \int_{}^{} \rho^{k+1}\overline{Z}^{k+1}kcZ^k | \overline{\theta}^{k+1} | \le \| \sqrt{\rho^{k+1}}\|_{L^\infty} \| \sqrt{\rho^{k+1}} \overline{Z}^{k+1}\|_{L^2}\| \sqrt{\rho^{k+1}} \overline{\theta}^{k+1}\|_{L^2} \\ & \le C_1 \| \sqrt{\rho^{k+1}} \overline{Z}^{k+1}\|^2_{L^2}+C_2\| \sqrt{\rho^{k+1}} \overline{\theta}^{k+1}\|^2_{L^2}.
\end{split}
\end{equation*}
Now we set $$L^k(t)= \overline{c}(\| \nabla {Z^k}\|^2_{L^2} + \|\nabla{Z^k}\|^2_{L^4})$$ which has the following properties
 $$L^k(t) \in L^1(0, T_1),  \quad \int_{0}^{t} L^k(s)ds\le \tilde{c},  \quad 0\le t \le T_1.$$
Summing up all the estimates of the $I_j$, $j=1,...6$ we get 
\begin{equation}\label{INEQZK}
\begin{split}
& \dfrac{d}{dt} \| \sqrt{\rho^{k+1}}\overline{Z}^{k+1}\|^2_{L^2}+ D \| \nabla {\overline{Z}^{k+1}}\|^2_{L^2} \\ & \le L^k(t)( \| \overline{\rho}^{k+1}\|^2_{L^2}+\| \sqrt{\rho^{k+1}}\overline{Z}^{k+1}\|^2_{L^2})+ \tilde{c} \| \nabla {\overline{u}^k}\|^2_{L^2} + c\int_{}^{} | \overline{\rho}^{k+1}| |Z^k_t| | \overline{Z}^{k+1}| dx
\end{split}
\end{equation}
We multiply \eqref{DIFFERENCE TEMP} by $\overline{\theta}^{k+1}$
\begin{equation}
\begin{split}
& \dfrac{1}{2}\dfrac{d}{dt}\int_{}^{} \rho^{k+1} | \overline{\theta}^{k+1}|^2 dx+ k\int_{}^{} | \nabla {\overline{\theta}^{k+1}}|^2 dx \\ & \le c\int_{}^{} [( | \nabla {u}^k|+ | \nabla {u}^{k-1}|) | \nabla {\overline{u}^k} | | \overline{\theta}^{k+1}|+| \overline{\rho}^{k+1}| | \theta^k_t| | \overline{\theta}^{k+1}|  \\ &+  \overline{\rho}^{k+1}qk\phi(\theta^k)|Z^k|  |\overline{\theta}^{k+1}| +\rho^{k+1}qk\overline{Z}^{k+1}c|\overline{\theta}^{k+1}|^2 \\ & +\rho^{k+1}qk\phi(\theta^k)\overline{Z}^{k+1}|\overline{\theta}^{k+1}|+\rho^{k+1}qk\phi(\theta^k)Z^k|\overline{\theta}^{k+1}| \\ & +| \overline{\rho}^{k+1}| [ | u^{k-1} | | \nabla {\theta}^k|+ | \theta^k| | \operatorname{div}u^{k-1}|] |\overline{\theta}^{k+1}| \\ & + \rho^{k+1}( | \overline{u}^k| | \nabla {\theta}^k| + | \overline{\theta}^{k+1}| | \operatorname{div}u^k|+| \theta^k| |\operatorname{div}\overline{u}^k|)| \overline{\theta}^{k+1}| \\ & \le c\int_{}^{} [( | \nabla {u}^k|+ | \nabla {u}^{k-1}|) | \nabla {\overline{u}^k} | | \overline{\theta}^{k+1}|+| \overline{\rho}^{k+1}| | \theta^k_t| | \overline{\theta}^{k+1}| +\sum_{j=1}^9 I_j.
\end{split}
\end{equation}
\begin{equation*}
\begin{split}
I_1 & \le c \int_{}^{} \overline{\rho}^{k+1} \sqrt{\theta^k}Z^k| |\overline{\theta}^{k+1}| \le c \| \overline{\rho}^{k+1} \|_{L^2} \| \theta^k\|_{L^2}^{\frac{1}{2}} \| \overline{\theta}^{k+1} \|_{L^4} \\ &   \le C_1 \| \overline{\rho}^{k+1} \|^2_{L^2}+C_2\| \nabla{\overline{\theta}^{k+1}} \|^2_{L^2},
\end{split}
\end{equation*}
\begin{equation*}
\begin{split}
I_2 & \le c \int_{}^{} \rho^{k+1} \overline{Z}^{k+1} |\overline{\theta}^{k+1}|^2  \le \| \sqrt{\rho^{k+1}}\overline{\theta}^{k+1}\|^2_{L^2},
\end{split}
\end{equation*}
\begin{equation*}
\begin{split}
I_3 & \le c \int_{}^{} \rho^{k+1}|\overline{\theta}^{k+1}| \sqrt{\theta^k} | \overline{Z}^{k+1} |  \le  \| \sqrt{\rho^{k+1}} \overline{\theta}^{k+1} \|_{L^2} \| \theta^k \|^{\frac{1}{2}}_{L^2} \| \nabla{\overline{Z}^{k+1}}\|_{L^2} \\ &  \le  C_1 \| \sqrt{\rho^{k+1}} \overline{\theta}^{k+1} \|_{L^2}^2+C_2\| \nabla{\overline{Z}^{k+1}}\|_{L^2}^2,
\end{split}
\end{equation*}
\begin{equation*}
\begin{split}
I_4 & \le c \int_{}^{}\rho^{k+1}Z^k|\overline{\theta}^{k+1}|^2  \le C_1\| \sqrt{\rho^{k+1}} \overline{\theta}^{k+1}\|_{L^2}^2,
\end{split}
\end{equation*}
\begin{equation*}
\begin{split}
I_5=\int_{}^{}| \overline{\rho}^{k+1}|  | u^{k-1} | | \nabla {\theta}^k| | \overline{\theta}^{k+1}| \le C \| \overline{\rho}^{k+1} \|_{L^2} \| \nabla {u^{k-1}} \|_{H^1} \| \nabla {\theta}^k \|_{H^1} \| \overline{\theta}^{k+1} \|_{L^2}.
\end{split}
\end{equation*}
For $I_5$ we point out that we use Morrey inequality combined with H\"{o}lder inequality,  in order to get an estimate in terms of the $H^1$ norm of $ \nabla {u^{k-1}}$ and $ \nabla {\theta}^k $.
\begin{equation*}
\begin{split}
I_6=\int_{}^{} \overline{\rho}^{k+1}| \theta^k| | \operatorname{div}u^{k-1}| |\overline{\theta}^{k+1}| \le C \| \nabla u^k \|_{L^\infty} \| \sqrt{\rho^{k+1}}\overline{\theta}^{k+1} \|^2_{L^2},
\end{split}
\end{equation*}
\begin{equation*}
\begin{split}
 I_7+I_8+I_9 & =\int_{}^{}\rho^{k+1}( | \overline{u}^k| | \nabla {\theta}^k| + | \overline{\theta}^{k+1}| | \operatorname{div}u^k|+| \theta^k| |\operatorname{div}\overline{u}^k|)| \overline{\theta}^{k+1}| \\ & \le  C \| \sqrt{\rho}^{k+1} \|_{L^\infty} \| \nabla {\overline{u}^k} \|_{L^2} \| \nabla \theta^k\|_{H^1} \| \sqrt{\rho}^{k+1}\overline{\theta}^{k+1}\|_{L^2}.
\end{split}
\end{equation*}
Summing up the estimates for the $I_j$ and observing that 
\begin{equation*}
\begin{split}
\int_{}^{} [( | \nabla {u}^k|+ | \nabla {u}^{k-1}|) | \nabla {\overline{u}^k} | \overline{\theta}^{k+1}| \le C( \| \nabla{u}^k \|_{L^3} + \| \nabla {u}^{k-1} \|_{L^3})\| \nabla {\overline{u}^k} \|_{L^2}| \nabla {\overline{\theta}^k} \|_{L^2}
\end{split}
\end{equation*}
 we obtain the following inequality
\begin{equation}\label{INEQTEMPK}
\begin{split}
& \dfrac{1}{2}\dfrac{d}{dt} \int_{}^{} \rho^{k+1} | \overline{\theta}^{k+1}|^2 dx+ k\| \nabla {\overline{\theta}^{k+1}}\|^2_{L^2} \\ & \le B^k (t)( \| \overline{\rho}^{k+1}\|^2_{L^2}+\| \sqrt{\rho^{k+1}}\overline{\theta}^{k+1}\|^2_{L^2}) + \tilde{C} \| \nabla {\overline{u}^k}\|^2_{L^2} \\ &+ C\int_{}^{} | \overline{\rho}^{k+1}| |\theta^k_t| | \overline{\theta}^{k+1}| dx
\end{split}
\end{equation}
\\
for some $B^k(t) \in L^1(0, T_1)$ and $\int_{0}^{t} B^k(s)ds\le \tilde{C}$ for all $k\ge 1$ and $t \in [0, T_1].$
\\
We multiply \eqref{DIFFERENCE MOMENT} by $\overline{u}^{k+1}$  and we integrate over $\Omega$, 

\begin{equation}
\begin{split}
& \dfrac{1}{2}\dfrac{d}{dt}\int_{}^{} \rho^{k+1}| \overline{u}^{k+1} |^2 dx+ \mu \int_{}^{} | \nabla {\overline{u}^{k+1}}|^2 dx \\ & \le   C\int_{}^{} |\overline{\rho}^{k+1}|( |f^{k+1}|+|u^k_t|+|u^{k-1} \cdot \nabla {u}^k|) | \overline{u}^{k+1}| \\ &+ \int_{}^{} \rho^{k+1} | \overline{u}^k| |\nabla {u}^k| | \overline{u}^{k+1}| +\ (| \rho^{k+1}| | \overline{\theta}^{k+1}|+| \overline{\rho}^{k+1}| | \theta^k|) |\nabla {\overline{u}^{k+1}}|] \\ & + \int_{}^{} \rho^k \overline{f}^{k+1} \overline{u}^{k+1}dx.
\end{split}
\end{equation}
\\
We recall that $\overline{f}^{k+1}$ depends on $\overline{\rho}^{k+1}$ via  the representation formula in terms of Poisson Kernel for the equation of the gravitational potential in \eqref{pb1}.
The contribution due to $f^{k+1}$ can be estimated by using the $L^\infty$ bound of $\rho^{k+1}$.
The following estimates hold
\begin{equation*}
\begin{split}
& \int_{}^{} | \overline{\rho}^{k+1} | |f^{k+1} | | \overline{u}^{k+1}| \le \| \overline{\rho}^{k+1} \|_{L^2} \| f^{k+1} \|_{L^4} \| \overline{u}^{k+1} \|_{L^4}  \\ & \le  C_1  \| \overline{\rho}^{k+1} \|_{L^2}^2 + C_2 \| \nabla{\overline{u}^{k+1}} \|_{L^2}^2,
\end{split}
\end{equation*}
\begin{equation*}
\begin{split}
\int_{}^{} | \overline{\rho}^{k+1}| |u^{k-1}| |\nabla{u^k}|  \le c \| \overline{\rho}^{k+1}\|_{L^2}^2 \| \nabla{u^k}\|_{L^4}^2+ \| \nabla{\overline{u}^{k+1}}\|_{L^2}^2,
\end{split}
\end{equation*}
\begin{equation*}
\begin{split}
\int_{}^{} \rho^{k+1} | \overline{u}^k| | \nabla{u}^k | |\overline{u}^{k+1} | \le c\| \sqrt{\rho}^{k+1}\overline{u}^{k+1} \|_{L^2}^2 \| \nabla{u}^k\|^2_{L^4}+ \| \nabla{\overline{u}^k}\|^2_{L^2},
\end{split}
\end{equation*}
\begin{equation*}
\begin{split}
\int_{}^{} | \rho^{k+1} | | \theta^{k+1}| | \nabla{\overline{u}^{k+1}}|  \le C_1 \| \| \sqrt{\rho}^{k+1} \overline{\theta}^{k+1} \|_{L^2}^2+ C_2 \| \nabla{\overline{u}^{k+1}} \|_{L^2}^2,
\end{split}
\end{equation*}
\begin{equation*}
\begin{split}
 \int_{}^{} | \overline{\rho}^{k+1} | \theta^k | | \nabla{\overline{u}^{k+1}} | & \le \| \overline{\rho}^{k+1} \|_{L^\infty} \| \theta^k \|_{L^2} \| \nabla{\overline{u}^{k+1}} \|_{L^2} \le C_1 \| \theta^k \|^2_{L^2}+ \| \nabla{\overline{u}^{k+1}}\|^2_{L^2},
\end{split}
\end{equation*}
\begin{equation*}
\begin{split}
& \int_{}^{} \rho^k \overline{f}^{k+1} \overline{u}^{k+1} \le \| \rho^k \|_{L^2} \| \overline{f}^{k+1} \|_{L^4} \| \overline{u}^{k+1} \|_{L^4} \le C_1 \| \rho^k \|_{L^2} \|^2+ C_2 \| \nabla{\overline{u}^{k+1}} \|_{L^2}^2,
\end{split}
\end{equation*}
\\
and similarly to \eqref{INEQZK},   \eqref{INEQTEMPK}, \eqref{INEQDK} we get
\begin{equation}\label{INEQMOMK}
\begin{split}
& \dfrac{d}{dt} \| \sqrt{\rho^{k+1}}\overline{u}^{k+1}\|^2_{L^2}+ \mu \| \nabla {\overline{u}^{k+1}}\|^2_{L^2}\\ & \le D^k_\eta (t)( \| \overline{\rho}^{k+1}\|^2_{L^2}+\| \sqrt{\rho^{k+1}}\overline{u}^{k+1}\|^2_{L^2}) \\ & + \tilde{C} \| \sqrt{\rho^{k+1}} \overline{\theta}^{k+1}\|^2_{L^2} + \eta \| \nabla {\overline{u}^k}\|^2_{L^2} + C\int_{}^{} | \overline{\rho}^{k+1}| |u^k_t| | \overline{u}^{k+1}| dx
\end{split}
\end{equation}
for some $D^k_{\eta}(t) \in L^1(0, T_1)$ and $\int_{0}^{t} D^k_{\eta}(s)ds\le \tilde{C}+ \tilde{C}_{\eta}t$ for all $k\ge 1$ and $t \in [0, T_1]$.
\\
We multiply equation \eqref{densitylk} by $\overline{\rho}^{k+1}$ and integrating over $\Omega$ we obtain
\begin{equation}
\begin{split}
& \dfrac{d}{dt}\int_{}^{} | \overline{\rho}^{k+1}|^2dx \\ & \le C \int_{}^{} ( | \nabla {u^k}| || \overline{\rho}^{k+1}|^2+ | \nabla {\rho^k}| | \overline{u}^k| | \overline{\rho}^{k+1}|+ \rho^k | \nabla {\overline{u}^k}|\overline{\rho}^{k+1}|)dx  \\ & \le  C( \| \nabla {u}^k \|_{W^{1,q}} \| \overline{\rho}^{k+1} \|^2_{L^2} + (\| \nabla {\rho}^k \|_{L^3}+ \| \rho^k\|_{L^\infty})\| \nabla {\overline{u}^k}\|_{L^2}\| \overline{\rho}^{k+1} \|^2_{L^2}).
\end{split}
\end{equation}
with $q\in (3,6].$
We define $$A^k_{\eta}(t)= C\| \nabla {u}^k \|_{W^{1,q}}+ \eta^{-1}C(\| \nabla {\rho}^k(t)\|^2_{L^3}+\| \rho^k(t)\|^2_{L^\infty}) $$
and the following inequality holds
\begin{equation}\label{INEQDK}
\dfrac{d}{dt}\| \overline{\rho}^{k+1} \|^2_{L^2}\le A^k_{\eta}(t)\| \overline{\rho}^{k+1} \|^2_{L^2}+ \eta \| \nabla {\overline{u}^k} \|^2_{L^2}.
\end{equation}
with $A^k_{\eta}(t) \in L^1(0, T_1)$ and $\int_{0}^{t} A^k_{\eta}(s)ds\le \tilde{C}+ \tilde{C}_{\eta}t$ for all $k\ge 1$ and $t \in [0, T_1]$ 
Now we set
\begin{equation}
\begin{split}
\psi^{k+1}(t)& = \| \overline{\rho}^{k+1}\|^2_{L^2}+ \dfrac{\eta}{k}\| \sqrt{\rho^{k+1}} \overline{\theta}^{k+1}\|^2_{L^2} \\&+ \| \sqrt{\rho^{k+1}}\overline{u}^{k+1}\|^2_{L^2}+\| \sqrt{\rho^{k+1}}\overline{Z}^{k+1}\|^2_{L^2}.
\end{split}
\end{equation}
Summing up \eqref{INEQZK},  \eqref{INEQTEMPK}, \eqref{INEQMOMK},  \eqref{INEQDK},  we obtain the following inequality
\begin{equation}\label{INEQUALITIPSI}
\begin{split}
& \dfrac{d}{dt} \psi^{k+1}+\eta \| \nabla {\overline{\theta}^{k+1}}\|^2_{L^2}+\eta \| \nabla {\overline{u}^{k+1}}\|^2_{L^2}+D \| \nabla {\overline{Z}^{k+1}}\|^2_{L^2} \\ & \le E^k_\eta (t)\psi^{k+1}+4\eta\tilde{C}\eta \| \nabla {\overline{u}^k}\|^2_{L^2} \\ &+ C \int_{}^{} | \overline{\rho}^{k+1}|( \eta |\theta^k_t| | \overline{\theta}^{k+1}| +|u^k_t| | \overline{u}^{k+1}| +|Z^k_t| | \overline{Z}^{k+1}| )dx,
\end{split}
\end{equation}
for some $E^k_{\eta}(t) \in L^1(0, T_1)$ and $\int_{0}^{t} E^k_{\eta}(s)ds\le \tilde{C}+ \tilde{C}_{\eta}t$ for all $k\ge 1$ and $t \in [0, T_1],$  
with $E^k_{\eta}(t)$ defined summing up $A^k_{\eta}(t),  B^k(t), D^k_{\eta}(t), L^k(t) .$
We estimate the integral in the right hand side of \eqref{INEQUALITIPSI} in terms of $\psi^{k+1}(t)$ and the quantities on the left hand side.  Note that since we are working on a bounded domain,  we don't need to refine the estimates for the density as in \cite{Cho 2006 bis} and we can compute easily the estimate.  We start with the term
\begin{equation}
 \int_{}^{} | \overline{\rho}^{k+1}|( \eta |\theta^k_t| | \overline{\theta}^{k+1}| +|u^k_t| | \overline{u}^{k+1}| +|Z^k_t| | \overline{Z}^{k+1}| )dx= \sum_{i=1}^3 J_i
\end{equation}
We point out that by construction, each $J_i,  i=1,  ..... , 3$ has a similar structure and so they can be easily estimated as follows
\begin{equation*}
\begin{split}
J_1 =\int_{}^{} | \overline{\rho}^{k+1}| \eta |\theta^k_t| | \overline{\theta}^{k+1}|  \le  \tilde{c} (\| \overline{\rho}^{k+1}\|^2_{L^2}\| \nabla {\theta}^k_t\|^2_{L^2}+\| \nabla {\overline{ \theta}^{k+1}}\|^2_{L^2},
\end{split}
\end{equation*}
\begin{equation*}
\begin{split}
J_2=\int_{}^{} | \overline{\rho}^{k+1}|  |u^k_t| | \overline{u}^{k+1}|  \le  \tilde{c} (\| \overline{\rho}^{k+1}\|^2_{L^2}\| \nabla {u}^k_t\|^2_{L^2}+\| \nabla {\overline{u}^{k+1}}\|^2_{L^2},
\end{split}
\end{equation*}
\begin{equation*}
\begin{split}
J_3=\int_{}^{} | \overline{\rho}^{k+1}|  |Z^k_t| | \overline{Z}^{k+1}|  \le  \tilde{c} (\| \overline{\rho}^{k+1}\|^2_{L^2}\| \nabla {Z}^k_t\|^2_{L^2}+\| \nabla {\overline{Z}^{k+1}}\|^2_{L^2},
\end{split}
\end{equation*}
where we observe that in all the estimates of the $J_i,$ $i=1, ...... , 3$ the first terms on the right hand side is estimated by $\psi^{k+1},$ while the second terms can be absorbed by the left hand side  of \eqref{INEQUALITIPSI} by choosing properly the Young inequality constant.  \\
Finally, we get 
\begin{equation}\label{FINAL INEQ PSI}
\begin{split}
& \dfrac{d}{dt} \psi^{k+1}+\eta \| \nabla {\overline{\theta}^{k+1}}\|^2_{L^2}+\mu \| \nabla {\overline{u}^{k+1}}\|^2_{L^2}+D \| \nabla {\overline{Z}^{k+1}}\|^2_{L^2} \\ & \le F^k_\eta (t)\psi^{k+1}+4\eta\tilde{C} \| \nabla {\overline{u}^k}\|^2_{L^2}
\end{split}
\end{equation}
for some $F^k_{\eta}(t) \in L^1(0, T_1)$ and $\int_{0}^{t} F^k_{\eta}(s)ds\le \tilde{C}+ \tilde{C}_{\eta}t$ for all $k\ge 1$ and $t \in [0, T_1]$.
Since $\psi^{k+1}(0)=0,$ by applying Gronwall inequality we get
\begin{equation}\label{3.93}
\begin{split}
& \psi^{k+1}(t) + \int_{0}^{t} (\eta \| \nabla {\overline{\theta}^{k+1}}\|^2_{L^2}+\mu \| \nabla {\overline{u}^{k+1}}\|^2_{L^2}+D \| \nabla {\overline{Z}^{k+1}}\|^2_{L^2}) \\ & \le  8\eta \tilde{C} \int_{0}^{t}\| \nabla {\overline{u}^k}\|^2_{L^2}ds) \exp{[\int_{0}^{t}F^k_{\eta}(s)ds]} \\ & \le 8\eta \tilde{C} \int_{0}^{t}\| \nabla {\overline{u}^k}\|^2_{L^2}ds) \exp{[C_1+C_2 t]}
\end{split}
\end{equation}
\\
By choosing the constants in such a way that $8\eta \tilde{C}e^{C_1}=1/2$, $e^{C_2 T_3}= \mu$,  we can take the $\sup$  in time in \eqref{3.93}  and summing over all $k\ge 1$  we end up with the following inequality
\begin{equation}\label{CAUCHY}
\begin{split}
& \sum_{k=1}^{\infty} \sup \limits_{0 \le t \le T_*} \psi^{k+1}(t) \\ &+\sum_{k=1}^{\infty} \int_{0}^{T_*}(\eta \| \nabla {\overline{\theta}^{k+1}}\|^2_{L^2}+\mu \| \nabla {\overline{u}^{k+1}}\|^2_{L^2}+D\| \nabla {\overline{Z}^{k+1}}\|^2_{L^2})dt \le \tilde{C} < \infty,
\end{split}
\end{equation}
with $T_*= \min(T_2, T_3).$
\\
Hence the approximating solution $(\rho^k,  \theta^k, u^k, Z^k)$ is convergent to a limit $(\rho,  \theta, u, Z)$ in the following strong sense:
\begin{equation}
\begin{cases}
\rho^k-\rho^1 \rightarrow{\rho-\rho^1} \quad in \quad L^\infty( 0, T_*; L^2), \\
(\theta^k, u^k, Z^k) \rightarrow {(\theta, u, Z)} \quad in \quad L^2( 0, T_*; H^1_0).
\end{cases}
\end{equation}
Moreover we get from \eqref{65} that the limit $(\rho,  \theta, u, Z)$ satisfy the following regularity estimate
\begin{equation}\label{REGULARITY NL}
\begin{split}
& \sup \limits_{0\le t\le T_1} (| \rho (t)|_{H^1 \cap W^{1, q}} +| \rho_t(t)|_{L^2 \cap L^q}) + \sup\limits_{0\le t\le T_1}  |( \theta,  u,  Z)(t)|_{H^1_{0} \cap H^2}) \\ & +  ess \sup \limits_{0\le t\le T_1} (\sqrt{\rho}\theta_t, \sqrt{\rho}u_t,  \sqrt{\rho}Z_t|)_{L^2} \\ & + \int_{0}^{T_1} (|( \theta_t,  u_t, Z_t)(t)|^2_{H^1_{0}}+(|( \theta_t,  u_t, Z_t)(t)|^2_{W^{2,q}}dt \le \tilde{C}.
\end{split}
\end{equation}
and it is a solution of the nonlinear problem \eqref{pb1}. \\
Regarding the uniqueness of the solution,  we proceed similarly to the proof of the uniqueness for the solution of the linearized system. 
We assume by contradiction that there exists two solutions $(\rho_1, \theta_1,  u_1,  Z_1) $ and $(\rho_2, \theta_2,  u_2,  Z_2) ,$we take the difference and we define: $$ \overline{\rho}=\rho_1-\rho_2,$$ $$\overline{\theta}=\theta_1-\theta_2,$$
$$ \overline{u}=u_1-u_2,$$ $$ \overline{Z}=Z_1-Z_2.$$ Then by using the same arguments as in  \eqref{INEQZK},   \eqref{INEQTEMPK}, \eqref{INEQMOMK},  \eqref{INEQDK}  and by applying Gronwall Lemma we can show that $$\overline{\rho}= \overline{\theta}= \overline{u}= \overline{Z}=0$$ in $[0, T_*] \times {\Omega},$ 
which proves the uniqueness of the solution.  Also the continuity in time of the solution $(\rho,  \theta, u, Z)$ follows with the same argument of the linearized system.
\section{Blow up criteria}
This Section is dedicated to the proof of Theorem \ref{firstblowup} and Theorem \ref{scndblowup}.
Both proofs are based on a contradiction argument.
Theorem \ref{firstblowup} follows from the estimates we proved in the local existence Theorem \ref{Teo2.1},  while the result in Theorem \ref{scndblowup} is a blow up criterion of the Beale Kato Majda type,  since we have that $\nabla u$ plays an important role in the loss of smoothness.  The blow up result in Theorem \ref{scndblowup} is more refined with respect to the one in Theorem \ref{firstblowup},  since it does not involve the density,  however,  an additional viscosity assumption is required. 
This result is proved by recovering some regularity estimates due to the contradiction hypothesis and by applying an extension argument.
\subsection{Proof of the Theorem 1.2}
Let $T_*$ be the maximal existence time of the solution $(\rho,  \theta, u, Z)$ of the nonlinear problem \eqref{pb1}. \\
We define the following functions 
\begin{equation}\label{Jt}
\begin{split}
J(t)=1+\|\rho(t)\|_{H^{1} \cap W^{1,q}}+ \|\rho_t(t)\|_{L^2\cap L^q}+\|(\theta(t), u(t), Z(t))\|_{H_{0}^1 \cap H^2}\\
 \|(\sqrt{\rho}\theta(t), \sqrt{\rho}u(t), \sqrt{\rho}Z(t))\|_{L^2}+{\|(\theta(t), u(t), Z(t))\|}^2 _{W^{2,q}}\\
+ \int_{0}^{t}  |(\theta_t(t), u_t(t), Z_t(t))\|_{H_{0}^1} \;,
 \end{split}
\end{equation}
\begin{equation}
\phi(t)=1+\|\rho(t)\|_{H^{1} \cap W^{1,q}}+\|(\theta(t), u(t), Z(t))\|_{H_{0}^1} \;.
\end{equation}
The function $J(t)$ represents the regularity class of the solution $(\rho,  \theta, u, Z)$,  while $\phi(t)$ is the quantity we want to show is unbounded.  
Let  $(\rho,  \theta, u, Z)$ be the solution in $[\tau,  T_*) \times \Omega$  with 
$\tau \in (0, T_*), $ $\tau$ fixed,  our goal is to obtain an inequality of the following form
\begin{equation}\label{KEY INEQ}
J(t)\le J(\tau)H(\phi(t)),
\end{equation}
with $\tau \in (0, T_*) ,$   and $H : [0, \infty) \rightarrow  [0, \infty), $ is an appropriate function enough regular and increasing with respect to $\phi .$ 
In order to get \eqref{KEY INEQ} we need to estimate all the terms in $J(t)$ by means of the function $\phi(t)$ or more in general by $H(\phi(t)).$ 
These estimates and the properties of $H$ are the key points of our argument. \\
Once the inequality \eqref{KEY INEQ} has been obtained,  we can argue by contradiction in the following way. \\
We assume that $$\lim_{t \rightarrow T_*} J(t)= L< \infty,$$
hence because of the continuity in time of the solution,  we have that $$J(T_*) < \infty .$$ 
If so,  we can consider $J(T_*)$ as a new initial data and we can apply Theorem \ref{Teo2.1} to extend the solution $(\rho, \theta, u, Z)$. This leads to a contradiction with respect to the maximality of $T_*.$
Hence,  we necessarily have that 
$$\lim_{t \rightarrow T_*} J(t)= \infty$$ 
and consequently also the right hand side in the inequality \eqref{KEY INEQ} is blowing up as $t \rightarrow{T_*}.$ Finally, thanks to the regularity of the function $H$, we have our claim.
\\
Now,  by following the same ideas of \cite{cho2004},  we estimate all of the elements of $J(t).$ 
For the density we have 
\begin{equation}\label{1}
\| \rho(t)\|_{H^1 \cap W^{1,q}} \le Cc_0 \exp {(c \int_{0}^{t} \| \nabla u \|_{H^1 \cap W^{1,q}})}
\end{equation} 
and from \eqref{reg U} it follows 

\begin{equation}\label{2}
\| \nabla u \|_{H^1} \le c(1+ \| \sqrt{\rho}u_t(t)\|^2_{L^2})H(\phi(t))
\end{equation}
\begin{equation}\label{3}
\| \nabla u \|_{W^{1,q}} \le c(1+ \| \sqrt{\rho}u_t(t)\|^2_{L^2})H(\phi(t)) + \| f \|_{L^q}+ \| \nabla {u_t (t)}\|_{L^2}.
\end{equation}
for a proper $H$ with the previously defined properties.
Similarly,  we rewrite the regularity estimates used to get \eqref{65} depending on $\phi(t)$ or $H(\phi(t).$
We write only the new estimates for the mass fraction $Z,$ all other estimates are obtained similarly and can be found in \cite{cho2004},  \cite{Cho 2006 bis}.
The regularity estimates for $Z$ can be written as follows
\begin{equation}\label{6}
\begin{split}
\| \sqrt{\rho}Z_t\|^2_{L^2}+ \int_{\tau}^{t} \| \nabla {Z_t} \|^2_{L^2}ds & \le c + c \| \sqrt{\rho}Z_t(\tau) \|^2_{L^2} \\ & + c \int_{\tau}^{t}( 1+ \| | \sqrt{\rho}Z_t(t)\|^2_{L^2}) H(\phi(t)),
\end{split}
\end{equation}
which leads to the following key inequality
\begin{equation}\label{STAR}
\begin{split}
\| \sqrt{\rho}u_t\|^2_{L^2}+ \int_{0}^{t} \| \nabla {u_t} \|^2_{L^2}ds+\| \sqrt{\rho}\theta_t\|^2_{L^2}+ \int_{0}^{t} \| \nabla {\theta_t} \|^2_{L^2}ds \\ +\| \sqrt{\rho}Z_t\|^2_{L^2}+ \int_{0}^{t} \| \nabla {Z_t} \|^2_{L^2}ds \le  cJ(\tau) \exp( cT_* \sup \limits_{0\le s \le t} H(\phi(s)).
\end{split}
\end{equation}
indeed by using $J(t)$ and the properties of $H$ 
\begin{equation}\label{S3}
\begin{split}
& \| \sqrt{\rho}Z_t\|^2_{L^2}+ \int_{0}^{t} \| \nabla {Z_t} \|^2_{L^2}ds \le cJ(\tau)\exp(c\int_{\tau}^{t} H(\phi(s))ds \\ & \le  cJ(\tau)\exp(c\int_{\tau}^{T_*}  \sup \limits_{0 \le s \le t} H(\phi(s))ds \le \tilde{c}J(\tau) \exp( cT_* \sup \limits_{0\le s \le t} H(\phi(s)).
\end{split}
\end{equation}
\\
and so \eqref{STAR} holds.  
\\
We have that all the elements of the regularity class \eqref{Jt} can be estimated properly depending on $\phi(t)$ or more in general on $H(\phi(t)),$ hence,  the following inequality holds
\begin{equation}
J(t)\le c J(\tau)( \sup \limits_{0\le s \le t}H(\phi(s))\exp(cT_* \sup \limits_{0\le s \le t} H(\phi(s))
\end{equation}
and using the contradiction argument previously discussed we prove our claim.
\subsection{Proof of the Theorem 1.3}
Let $T_*$ be the maximal existence time of the solution $(\rho,  \theta, u, Z)$ of the nonlinear problem \eqref{pb1}.  \\
We assume by contradiction that for any $T < T^*$
\begin{equation}\label{ASSURDO 2}
(\|\theta\|_{L^\infty(0, T, L^\infty)}+\| \nabla u\|_{L^1(0, T; L^\infty)}+\| \nabla Z\|_{L^1(0, T; L^\infty)})\le c< \infty.
\end{equation}
\\
We recall that in the framework of \cite{fan} we require the following additional hypothesis on the viscosity coefficients 
\begin{equation}\label{visk cap4}
\mu>\frac{1}{7}\lambda.
\end{equation}
Moreover we consider the following initial and boundary conditions
\begin{equation}\label{bc cap 4}
\begin{split}
\begin{cases}
(\rho, \theta, u, Z)_{|t=0}=(\rho_0, \theta_0, u_0, Z_0) \quad in \quad \Omega \\
u=0,\quad \dfrac{\partial\theta}{\partial\nu}=0,  \quad Z=0 \quad on\quad \partial\Omega,
\end{cases}
\end{split}
\end{equation}
where $\nu$ is the outgoing unitary normal to $\partial\Omega .$ \\
Our first claim is that,  the density $\rho$ is bounded from above.  Indeed by using characteristics method,   Gronwall inequality and \eqref{ASSURDO 2},  we have that for any $x \in \overline{\Omega},$ $t\in [0, T]$ the following bounds hold
\begin{equation}\label{above bound density}
0 \le \underline{\rho}\exp \left(-\int_{0}^{t} \| \operatorname{div}u \|_{L^\infty}dt \right) \le \rho(x,t) \le \overline{\rho}\exp \left(-\int_{0}^{t} \| \operatorname{div}u \|_{L^\infty}dt \right) \le C
\end{equation}
where $0 \le \underline{\rho} \le \rho_0 \le \overline{\rho}.$
\\
Our next claim is that the temperature $\theta$ is positive a.e in $[0, T] \times \Omega.$  In order to prove it we use the following auxiliary function
\begin{equation}\label{AUX H}
H(\theta)(x,t):=c_v \max \{-\theta(x,t), 0\}.
\end{equation}
We deduce from \eqref{AUX H} that
\begin{equation}\label{H'}
H'(\theta)=\begin{cases} -c_v,  \quad  \theta < 0 \\ 0,  \quad \quad  \theta > 0
\end{cases}
\end{equation}
and 
\begin{equation}\label{H''}
H''(\theta)=0.
\end{equation}
We take the time evolution equation for the temperature in \eqref{pb1} and we multiply it by $H'(\theta)$ and integrate over $\Omega$ and we get
\begin{equation}
\begin{split}
& \int_{\Omega} \rho(H(\theta)_t+u \cdot \nabla {H(\theta)})+ R\rho H(\theta)\operatorname{div}u)dx \\ & = \int_{\Omega} H'(\theta)(k\Delta \theta+\dfrac{\mu}{2}| \nabla u+ \nabla {u^T} |^2 +\lambda(\operatorname{div}u)^2+\rho qk\phi(\theta)Z)dx \\ & \le\int_{\Omega} H'(\theta) \rho qk\phi(\theta)Zdx \le 0.
\end{split}
\end{equation}
Moreover we take the continuity equation in \eqref{pb1} and we multiply it by $H(\theta)$ and integrate over $\Omega$,  we obtain
\begin{equation}\label{GRONW H}
\dfrac{d}{dt}\int_{\Omega} \rho H(\theta)dx \le \| \operatorname{div}u\|_{L^\infty} \int_{\Omega} \rho H(\theta)dx.
\end{equation}
By applying Gronwall inequality to \eqref{GRONW H} we have 
\begin{equation}
\int_{\Omega} \rho H(\theta)dx\equiv 0 \quad \text{for any}\; t\in [0, T],
\end{equation}
since,  $\theta_0\ge 0$.  Hence,  by the definition of $H(\theta)$ we necessarily deduce that $\theta \ge 0.$
We recall the following identity
\begin{equation}
\dfrac{\Delta {\theta}}{\theta}=\operatorname{div}\left(\dfrac{\nabla {\theta}}{\theta}\right)-\nabla \left({\dfrac{1}{\theta}}\right) \cdot \nabla \theta=\operatorname{div}\left(\dfrac{\nabla {\theta}}{\theta}\right)+\dfrac{|\nabla \theta |^2}{\theta^2}
\end{equation}
and define $s=\log\theta$ so that $s$ satisfies the following equation 
\begin{equation}\label{identity theta}
\begin{split}
& (\rho s)_t+\operatorname{div}(\rho su))-k\operatorname{div}\left(\dfrac{\nabla {\theta}}{\theta}\right)+R\rho \operatorname{div}u \\ & =\dfrac{1}{\theta}[\dfrac{\mu}{2}| \nabla u+\nabla {u^T}|^2 +\lambda(\operatorname{div}u)^2] + \dfrac{k}{\theta^2}| \nabla \theta|^2 + \dfrac{1}{\theta}\rho qk \phi(\theta)Z.
\end{split}
\end{equation}
Integrating the last equation in $\Omega \times (0,T)$ using \eqref{ASSURDO 2} and the initial conditions,  we get
\begin{equation}\label{alpha}
\begin{split}
& \int_{0}^{T} \int_{\Omega} \left( \dfrac{\alpha | \nabla u|^2}{\theta}+\dfrac{k | \nabla \theta|^2}{\theta^2} +\dfrac{1}{\theta}\rho qk \phi(\theta)Z  \right)dxds-\int_{\Omega}\rho \log \theta dx  \bigg{|}_{t=T} \\ & \le C \| \rho\|_{L^\infty(0, T; L^\infty)}\| \operatorname{div}u\|_{L^1(0, T; L^\infty)}+ \int_{\Omega} \rho_0 \log \theta_0 dx \le C,
\end{split}
\end{equation}
where $\alpha$ is a generic positive constant.  The second term in the left hand side of \eqref{alpha} can be estimated as follows,  
\begin{equation}
\int_{\Omega \cap \{\theta \ge 1 \}} \rho \log \theta dx \bigg{|}_{t=T} \le C \| \rho\|_{L^\infty(0, T; L^\infty)} \log \| \theta \|_{L^\infty(0, T; L^\infty)} \le C
\end{equation}
where we used the continuity of $\rho$ and $\theta$ with respect to time.  Using this estimate,  we can rewrite \eqref{alpha} in the following way
\begin{equation}
\begin{split}
&  \int_{0}^{T} \int_{\Omega} \left( \dfrac{\alpha | \nabla u|^2}{\theta}+\dfrac{k | \nabla \theta|^2}{\theta^2}+\dfrac{1}{\theta}\rho qk \phi(\theta)Z  \right)dxds \\ & +\int_{\Omega \cap \{0 \le \theta \le 1 \}} \rho \log \theta dx \bigg{|}_{t=T} \le C
 \end{split}
\end{equation}
We observe that the term due to the contribution of the mass fraction $Z$ can be estimated as follows 
\begin{equation}\label{new term Z}
\int_{0}^{T} \int_{\Omega}\dfrac{1}{\theta}\rho qk \phi(\theta)Z \le c \| \rho \|_{L^\infty(0,T; L^\infty)} \| Z \|_{L^\infty (0, T; L^\infty)} \le C
\end{equation}
since the function $\dfrac{\phi(\theta)}{\theta}$ is bounded.
\\
The estimates \eqref{identity theta}-\eqref{new term Z} are the proof of the following Lemma.
\begin{lemma}
\label{Lemma 2.1 cap4}
For any $T < T_*,$ we have 
\begin{equation}
\sup \limits_{0 \le t \le T} \int_{\Omega} \rho(t) | \log \theta(t) |dx+ \int_{0}^{T} \int_{\Omega}( | \nabla {\log \theta}|^2+ | \nabla u|^2+| \nabla \theta|^2)dxdt \le C.
\end{equation}
\end{lemma}

We now recall some Lemmas which will be useful in order to prove the positivity of $\theta.$

\begin{lemma}
\label{Lemma 2.2 cap4}
Let $\Omega$ be a bounded domain in $\mathbb{R}^N$ and $\gamma > 1$ be a constant.  Given constants $M$ and $E_0$ with $0 < M < E_0$,  there is a constant $C(E_0, M)$ such that for any non-negative function $\rho$ satisfying 
\begin{equation}
M \le \int_{\Omega} \rho dx, \quad  \int_{\Omega} \rho^{\gamma}dx \le E_0
\end{equation}
and any $v \in H^1(\Omega)$,
\begin{equation}
\| v \|_{L^2}^2 \le C \bigg{[} \| \nabla v \|_{L^2(\Omega)}^2+ \left( \int_{\Omega} \rho |v| dx \right)^2\bigg{]}.
\end{equation}
\end{lemma}
For a detailed proof we refer to \cite{feir}. \\
Combining Lemma \ref{Lemma 2.2 cap4}  with \eqref{above bound density}, we get 
\begin{equation}
\int_{0}^{T}\int_{\Omega} | \log \theta |^2 dxdt \le C.
\end{equation}
Moreover,  since $\theta \in C([0,T], H^2)$,  then $\theta$ and also $\log \theta$ are continuous with respect to space too.  In particular,  from last inequality we get that $| \log \theta |$  is bounded. By combining this property and the assumption on $\theta$ we get that $$\theta(x,t) >0, \quad \forall x \in \overline{\Omega}, \quad t\in [0,T].$$ which is our claim.\\
The next is a technical Lemma due to Hoff,  which is  strongly related to the viscosity assumption \eqref{visk cap4}.  We point out that the external force $f=-\nabla \Phi$  is enough regular in order to apply the lemma,  indeed it has the  regularity \eqref{reg f}. 

\begin{lemma}
\label{Lemma 2.3 cap4}
Let $7\mu > \lambda.$ Then there is a small $\delta > 0$,  such that 
\begin{equation}
\sup \limits_{0 \le t \le T} \int_{\Omega} \rho(x,t) | u(x,t) |^{3+\delta}dx+ \int_{0}^{T} \int_{\Omega}( | u |^{1+\delta} | \nabla u|^2|dxdt \le C.
\end{equation}
\end{lemma} 
We refer to \cite{fan} and \cite{huang2010}  for the proof. 
By using the Lemma \ref{Lemma 2.3 cap4} we estimate the first order spatial derivatives of the density $\rho$ and velocity $u$. 

\begin{lemma}
\label{Lemma 2.4 cap4}
Under \eqref{ASSURDO 2}, we have  for any $T < T_*$ that 
\begin{equation}
\sup \limits_{0 \le t \le T} \| \nabla {\rho(t)}\|_{L^2} + \int_{0}^{T} \| \rho_t \|_{L^2}^2 dt \le C,
\end{equation}
\begin{equation}
\sup \limits_{0 \le t \le T} \| u(t) \|_{H^1_0}^2 + \int_{0}^{T} \int_{\Omega} \rho |u_t|^2dxdt \le C,
\end{equation}
\begin{equation}\label{stima u H2 in t}
\int_{0}^{T} \|u(t)\|_{H^2}^2dt \le C.
\end{equation}
\end{lemma}

\begin{proof}
The proof follows the same line of argument of Lemma 2.4 in \cite{fan} with a slight difference due to our choice for the external force $f$ and heat source $h$.  For this purpose we recall that
\begin{comment}
We multiply the momentum equation in \eqref{pb1} by $u_t$ and we integrate over $\Omega$ 
\begin{equation}\label{2.14 fan}
\begin{split}
\dfrac{d}{dt} & \int_{\Omega} \left(\dfrac{\mu}{2}| \nabla u|^2 +\dfrac{\mu+ \lambda}{2}(\operatorname{div}u)^2 \right)dx+\dfrac{1}{2}\int_{\Omega} \rho |u_t|^2dx \\ & \le \int_{\Omega} \rho |u \cdot \nabla u|^2 dx- \int_{\Omega} \nabla P \cdot u_t dx+ \int_{\Omega} \rho f u_t dx.
\end{split}
\end{equation}
The estimates of the integrals in the right hand side of  \eqref{2.14 fan} are similar to the ones in \cite{fan},  so we omit all the details here.  We point out that here the second integral in the right hand side of \eqref{2.14 fan} can be rewritten as
\begin{equation}\label{2.16 fan}
\int_{\Omega} \nabla P \cdot u_t dx= \int_{\Omega}P_t\operatorname{div}u dx- \dfrac{d}{dt}\int_{\Omega}P\operatorname{div}udx.
\end{equation}
and the time evolution equation for the pressure reads as follows
\begin{equation}\label{pressione cap 4}
\begin{split}
P_t+u\cdot \nabla P+ \gamma P\operatorname{div}u & =(\gamma-1)\Delta \theta  +(\gamma-1)\left( \dfrac{\mu}{2}| \nabla u+ \nabla {u^T}|^2 + \lambda(\operatorname{div}u)^2\right) \\ & +(\gamma-1)\rho q k \phi(\theta)Z.
\end{split}
\end{equation}
\end{comment} 
the velocity $u$ is solution to the following elliptic system 
\begin{equation}
-\mu \Delta u -(\lambda+\mu)\nabla {\operatorname{div}u}=F
\end{equation}
with $F:=-\rho u_t-\rho u \cdot \nabla u-\nabla P+f, $ and the following elliptic regularity estimate holds
\begin{equation}\label{2.18 fan Elliptic}
\begin{split}
\| u \|_{H^2} & \le C(  \| \sqrt{\rho}u_t\|_{L^2}+ | \nabla \rho \|_{L^2}+ \| \nabla \theta \|_{L^2}+\|  \nabla u \|_{L^2}+ \| f \|_{L^2}).
\end{split}
\end{equation}
while,  concerning the estimate for the pressure we have 
\begin{equation*}
\int_{0}^{t} \int_{\Omega} \rho q k \phi(\theta)Z \le c \| \rho \|_{L^\infty(0,T; L^\infty)} \| \sqrt{\theta} \|_{L^\infty (0, T; L^\infty)}  \| Z \|_{L^\infty (0, T; L^\infty)} \le C
\end{equation*}
\end{proof}
The next Lemma guarantees high order estimates for the temperature $\theta$ and velocity $u$ which are essential to apply our contradiction argument.
\begin{lemma}
\label{Lemma:2.5_cap4}
Let $$\Phi(t):=1+\left( \int_{0}^{t} \| \theta_t(s)\|^2_{H^1}ds \right)^\frac{1}{2}$$
Then for any $T < T_*$,  we have 
\begin{equation}\label{claim1}
\sup \limits_{0 \le t \le T} \| \theta (t)\|^2_{H^1}+ \int_{0}^{T} \int_{\Omega} \rho \theta_t^2 dxdt \le C \Phi(T),
\end{equation}
\begin{equation}\label{claim2}
\sup \limits_{0 \le t \le T} \| u (t)\|^2_{H^2}+ \int_{0}^{T} \| \theta (t) \|^2_{H^2} dt \le C \Phi(T),
\end{equation}
\begin{equation}\label{claim3}
\sup \limits_{0 \le t \le T} \| \sqrt{\rho(t)}u_t(t)\|^2_{L^2}+ \int_{0}^{T} \|u_t(t) \|^2_{H^1_0}dt \le  C \Phi(T),
\end{equation}
\end{lemma}
\begin{proof}
We take the temperature equation in \eqref{pb1} and we multiply it by $\theta_t$ and we integrate over $\Omega$. Then by using the results of Lemma \ref{Lemma 2.1 cap4},  Lemma \ref{Lemma 2.4 cap4},  the elliptic regularity estimate \eqref{2.18 fan Elliptic},  H\"{o}lder,  Sobolev and Young inequalities,  we get 
\begin{equation}\label{Gronwall per grad theta}
\begin{split}
& \dfrac{k}{2}\dfrac{d}{dt} \int_{\Omega} | \nabla \theta|^2 dx + c_v \int_{\Omega} \rho \theta_t^2 dx   \\ & =-c_v \int_{\Omega} \rho(u\cdot \nabla)\theta \theta_t dx- \int_{\Omega} P\operatorname{div}u \theta_t dx \\ & + \int_{\Omega} \left(\dfrac{\mu}{2}| \nabla u + \nabla u^T|^2+ \lambda(\operatorname{div}u)^2 \right)\theta_t dx  +\int_{\Omega} \rho qk\phi(\theta)Z\theta_t \\ & \le C \| u \|_{L^\infty}\| \nabla \theta \|_{L^2} \| \sqrt{\rho}\theta_t\|_{L^2} \\ & +C \| \nabla u \|_{L^2}( \| \| \sqrt{\rho}\theta_t\|_{L^2}+ \| \nabla u \|_{L^3} \| \theta_t\|_{H^1}) + \| \sqrt{\rho}\theta_t\|_{L^2} \| Z \|_{L^2} \\ & \le C \| u \|_{H^2} \| \nabla \theta \|_{L^2}\| \sqrt{\rho}\theta_t\|_{L^2} \\ & +C\| \sqrt{\rho}\theta_t\|_{L^2}+C\| u \|^\frac{1}{2}_{H^2}\| \theta_t\|_{H^1} + \| \sqrt{\rho}\theta_t\|_{L^2} \| Z \|_{L^2} \\& \le \epsilon \| \sqrt{\rho}\theta_t\|^2_{L^2}+C(\epsilon)(1+\|u\|^2_{H^2} \| \nabla \theta \|^2_{L^2}+\| u \|^\frac{1}{2}_{H^2}\| \theta_t\|_{H^1} \| Z \|^2_{L^2})
\end{split}
\end{equation}
for any $0 < \epsilon < 1.$ 
Integrating in time,  since the terms in the right hand side of \eqref{Gronwall per grad theta} have finite $L^1$ norm in time,  by applying Gronwall inequality  we get \eqref{claim1}.\\
We prove now \eqref{claim2}.
We take the time derivative in the momentum equation in \eqref{pb1} and we multiply it by $u_t$ and we integrate over $\Omega$
\begin{equation}
\begin{split}
& \dfrac{1}{2}\dfrac{d}{dt} \int_{\Omega} \rho u_t^2 dx + \int_{\Omega}\left( \mu | \nabla {u_t}|^2+ (\lambda+\mu )(\operatorname{div}u_t)^2\right)dx \\ & =\int_{\Omega} P_t\operatorname{div}u_t dx - \int_{\Omega}\rho u \cdot \nabla [(u_t+u\cdot \nabla u)u_t]dx \\ & - \int_{\Omega} \rho u_t\cdot \nabla u \cdot u_t dx+ \int_{\Omega} \rho_t f u_t+ \int_{\Omega} \rho f_tu_t:=  \sum_{j=1}^5 I_j
\end{split}
\end{equation}
The estimate for the integrals $I_j, j=1,...,3$ are similar to \cite{fan},  for completeness we recall them here and then we compute the estimates for $I_4$ and $I_5$ which depend on the external force $f$
\begin{equation} \label{I1,I3}
\begin{split}
\hspace{-0,4cm}|I_1|\!+\!| I_2|\!+\! |I_3| & \le  \epsilon \| \nabla {u_t}\|^2_{L^2}+C(\epsilon)( \| \rho_t\|^2_{L^2}+\| \sqrt{\rho}\theta_t\|_{L^2})+C( \epsilon) \| u_t \|^2_{H^1}  \\ & + C\epsilon^{-1} \| u \|^2_{H^2} (C(\epsilon) +  \| \nabla u \|_{H^1} \| \nabla {u_t} \|^2_{H^1}).
\end{split}
\end{equation}
\begin{equation}
\begin{split}
I_4 & = \int_{\Omega} \rho_t f u_t \le \| \rho_t \|_{L^2} \|u_t \|_{L^6} \| f \|_{L^3} \\ & \le C_{\epsilon}\| \rho_t \|^2_{L^2}+\epsilon \|u_t\|^2_{H^1}
\end{split}
\end{equation}
\begin{equation}\label{I5 CAP 4}
\begin{split}
I_5=\int_{\Omega}\rho f_t u_t & \le \| \rho \|_{L^\infty} \int_{\Omega} f_t u_t \le c \left \langle f_t,  u_t \right  \rangle \\ & \le c \| f_t \|_{H^{-1}} \|u_t \|_{H^1} \\ & \le \| f_t \|^2_{H^{-1}}+ \|u_t \|^2_{H^1} \le C_1 +C_2 \|u_t \|^2_{H^1}
\end{split}
\end{equation}
These estimates leads to the following inequality
\begin{equation}
\begin{split}
& \dfrac{1}{2}\dfrac{d}{dt} \int_{\Omega} \rho u_t^2 dx + \int_{\Omega}\left( \mu | \nabla {u_t}|^2+ (\lambda+\mu )(\operatorname{div}u_t)^2\right)dx \\ & \le C ( \| \sqrt{\rho}u_t \|^2_{L^2}+ \| u \|^2_{H^2}+ \| \rho_t\|^2_{L^2})+ C_1 \| \sqrt{\rho}u_t \|^2_{L^2}.
\end{split}
\end{equation}
By applying Gronwall inequality and using \eqref{claim1} we get \eqref{claim3}.  To conclude the proof of Lemma \ref{Lemma:2.5_cap4} we observe that \eqref{claim2} follows by considering the time evolution equation for the temperature in \eqref{pb1} and using the elliptic regularity estimate \eqref{2.18 fan Elliptic} and Lemma \ref{Lemma 2.4 cap4}. \\
\end{proof}
The next lemma regards some estimates on $\theta_t.$
\begin{lemma}
\label{Lemma 2.6 cap4}
For any $T < T_*$, we have 
\begin{equation}
\sup \limits_{0 \le t \le T} \int_{\Omega} \rho(x,t) \theta^2_t(x,t)dx + \int_{0}^{T} \| \theta_t (t) \|^2_{H^1} dt \le C,
\end{equation}
\begin{equation}
\sup \limits_{0 \le t \le T} \| \theta (t)\|^2_{H^2} \le C.
\end{equation}
\end{lemma}
\begin{proof}
We consider the temperature equation in \eqref{pb1} and we take the time derivative and we multiply it by $\theta_t$. By integrating over $\Omega$ we have
\begin{equation}
\begin{split}
& \dfrac{1}{2}\dfrac{d}{dt}\int_{\Omega}\rho \theta_t^2dx+k\int_{\Omega} | \nabla {\theta_t}|^2dx \\ & = \int_{\Omega}R \rho \theta_t^2 \operatorname{div}u\theta_tdx+\int_{\Omega}R \rho_t \theta \operatorname{div}u\theta_tdx+R\rho \theta \operatorname{div}u_t\theta_t dx \\ & + \int_{\Omega} [ \mu ( \nabla u + \nabla u^t) : (\nabla {u_t}+ \nabla {u^t_t})+ 2\lambda \operatorname{div}u\operatorname{div}u_t] \theta_t dx \\ & - \int_{\Omega} \rho u \cdot \nabla \theta \theta_t dx -\int_{\Omega} \rho u_t \cdot \nabla \theta \theta_t- \int_{\Omega} \rho_t \theta^2_t dx+\int_{\Omega} \rho_t qk\phi(\theta)Z\theta_t \\ & +\int_{\Omega} \dfrac{1}{\theta^2}\rho qk\phi(\theta)\theta_t^2Z+ \int_{\Omega}\rho qk \phi(\theta)Z_t \theta_t := \sum_{i=1}^{10} J_i
\end{split}
\end{equation}

The estimates for the integrals $J_i$, $i=1,.....,7$ are quite standard,   we rewrite them here and we refer to \cite{fan} for a detailed discussion.  
\begin{equation}
\begin{split}
|J_1|+|J_2|+|J_3|+|J_7| & \le  C(\epsilon) \| \sqrt{\rho}\theta_t \|_{L^2} ( \| \nabla u \|^2_{H^1} + \| u \|^2_{H^2}) \\ & +  \epsilon \| \theta_t\|^2_{H^1}+ C(\epsilon) \| u_t \|^2_{H^1} 
\end{split}
\end{equation}
\begin{equation}
\begin{split}
\int_{0}^{t} |J_4|+|J_5|+|J_6| ds & \le C(\epsilon) \left( 1+ \| \theta_t \|^2_{L^2(0,t; H^1)}+ \int_{0}^{t} \|u \|^2_{H^2}\|\sqrt{\rho}\theta_t \|_{L^2} ds\right) 
\end{split}
\end{equation}
\begin{comment}
\begin{equation}
\begin{split}
\int_{0}^{t} |J_4| ds  \le C\epsilon^{-1}+ \epsilon \| \theta_t \|^2_{L^2(0,t;H^1)}.
\end{split}
\end{equation}
\begin{equation}
\begin{split}
\int_{0}^{t} |J_5| ds \le C\epsilon \| \theta_t\|^2_{L^2(0,t; H^1)}+C(\epsilon) \left( 1+ \int_{0}^{t} \|u \|^2_{H^2}\|\sqrt{\rho}\theta_t \|_{L^2} ds\right).
\end{split}
\end{equation}
\begin{equation}
\begin{split}
\int_{0}^{t} |J_6| ds \le C (\epsilon) + C \epsilon \| \theta_t \|^2_{L^2(0,t; H^1)}.
\end{split}
\end{equation}
\begin{equation}\label{J7 cap 4}
\begin{split}
|J_7|  \le \epsilon \| \theta_t \|^2_{H^1}+ C \epsilon^{-1} \| u \|^2_{H^2} \|\sqrt{\rho}\theta_t \|_{L^2},
\end{split}
\end{equation}
\end{comment}
We write the details for $J_8, J_9,J_{10}$ which depends on the Arrhenius function $\phi(\theta).$
We use Lemma \ref{Lemma 2.4 cap4} and we make an extensive use of H\"{o}lder,  Sobolev and Young inequalities.
\begin{equation}
\begin{split}
|J_8|= \int_{\Omega}\rho_t qk \phi(\theta)Z\theta_t & \le c \| \rho_t \|_{L^2} \| \theta_t \|_{L^4} \| \sqrt{\theta} \|_{L^4} \\ & \le \tilde{c} \| \rho_t\|^2_{L^2} + \tilde{C}(\epsilon) \| \theta_t \|^2_{H^1},
\end{split}
\end{equation}
\begin{equation}
|J_9|= \int_{\Omega} \dfrac{1}{\theta^2}\rho qk\phi(\theta)\theta_t^2Z \le c_1 \| Z \|_{L^\infty}\| \sqrt{\rho}\theta_t\|^2_{L^2} \le   c_2 \| \sqrt{\rho}\theta_t\|^2_{L^2},
\end{equation}
since the function $\dfrac{\phi(\theta)}{\theta^2} $ is bounded.
\begin{equation}\label{J10 cap 4}
\begin{split}
|J_{10}| =\int_{\Omega}\rho qk \phi(\theta)Z_t \theta_t &  \le C_1\| \sqrt{\rho}\theta_t\|^2_{L^2}+C_2 \| \sqrt{\rho} Z_t \|^2_{L^2}.
\end{split}
\end{equation}
We end up up the following inequality
\begin{equation}
\begin{split}
& \| \sqrt{\rho(t)}\theta_t(t) \|^2_{L^2}+ \| \theta_t \|^2_{L^2(0, T; H^1)}\\ &  \le C+ C \int_{0}^{t} ( 1+ \| u \|^2_{H^2}) \| \sqrt{\rho(s)}\theta_t(s) \|^2_{L^2} ds, 
\end{split}
\end{equation}
where $0 \le t \le T.$ By applying once more Gronwall inequality,  we get the first claim of Lemma \ref{Lemma 2.6 cap4}.  The estimate for the temperature can be easily obtained by considering the temperature equation in \eqref{pb1} and Lemma \ref{Lemma:2.5_cap4}. This concludes this proof.
\end{proof}
The next lemma is related to an additional $L^q$ regularity of the solution. 
We omit the proof since by taking into account the regularity \eqref{reg f} the proof follows the same line of arguments as in \cite{fan}.
\begin{lemma}
\label{Lemma 2.7 cap4}
Let $ q \in (3, 6].$ Then 
\begin{equation}
\sup \limits_{0 \le t \le T} ( \| \rho_t(t)\|_{L^q}+ \| \rho(t)\|_{W^{1, q}} \le C,
\end{equation}
\begin{equation}
\int_{0}^{t} ( \| u(t) \|^2_{W^{2,q}}+ \| \theta(t)\|^2_{W^{2, q}}) dt \le C
\end{equation}
\end{lemma}
We complete our discussion by stating and proving the regularity Lemmas concerning the mass fraction $Z$.  We point out that these estimates are the more relevant part of our contribution in the proof of this blow up result.
\begin{lemma}
\label{Lemma 2.8 cap4}
For any $T < T_*$ the following estimates hold
\begin{equation}
\sup \limits_{0 \le t \le T} \| Z \|^2_{H^1_0}+\int_{0}^{T} \int_{\Omega} \rho |Z_t|^2 dxdt \le C,
\end{equation}
\begin{equation}
\int_{0}^{T} \| Z \|^2_{H^2} dt \le C.
\end{equation}
\end{lemma}
\begin{proof}
We consider the equation for the mass fraction $Z$ in \eqref{pb1},  we multiply by $Z_t$ and we integrate over $\Omega$
\begin{equation}
\begin{split}
& \dfrac{d}{dt}\int_{\Omega} \dfrac{D}{2}|\nabla Z|+\int_{\Omega} \rho |Z_t|^2=\int_{\Omega} Z_t[ -\rho u\nabla Z- k\rho \phi(\theta)Z]   \\  & \le \int_{\Omega} |Z_t|\rho |u| |\nabla Z|+ \int_{\Omega}|Z_t| k \rho \phi(\theta)Z.
\end{split}
\end{equation}
We observe that 
$$\int_{\Omega} |Z_t|\rho |u| |\nabla Z| \le \dfrac{1}{2} \int_{\Omega} \rho |Z_t|^2+ \dfrac{1}{2} \int_{\Omega} \rho | u\cdot \nabla Z |^2,$$
so we get 
\begin{equation}
\begin{split}
& \dfrac{d}{dt}\int_{\Omega} \dfrac{D}{2}|\nabla Z|^2+\dfrac{1}{2}  \int_{\Omega} \rho |Z_t|^2=\int_{\Omega} Z_t[ -\rho u\nabla Z- k\rho \phi(\theta)Z]   \\  & \le \dfrac{1}{2} \int_{\Omega} \rho | u \cdot \nabla Z |^2+ \int_{\Omega}k \rho \phi(\theta)Z |Z_t|= I_1+ I_2.
\end{split}
\end{equation}
We start estimating the $I_j$.

\begin{equation}
\begin{split}
I_1=\dfrac{1}{2} \int_{\Omega} \rho | u \cdot \nabla Z |^2 &  \le  \left( \int_{\Omega} \rho |u|^q dx \right)^\frac{2}{q} \| \nabla Z \|^2_{L^{\frac{2q}{q-2}}}
\end{split}
\end{equation}
where we used H\"{o}lder inequality with exponents $p=\dfrac{q}{2}$ and $p'= \dfrac{q}{q-2}$ while for the first integral we used Lemma 4.3 with $q=3+ \delta.$
In order to estimate $\|\nabla Z \|^2_{L^{\frac{2q}{q-2}}}$ we  use the  Gagliardo Nirenberg interpolation inequality and we get
\begin{equation*}
\| \nabla Z \|_{L^{\frac{2q}{q-2}}} \le \| \nabla Z \|^{\frac{1}{2}}_{L^2}\| \nabla Z \|^{\frac{1}{2}}_{L^q},
\end{equation*}
so that
\begin{equation}
\begin{split}
\| \nabla Z \|^2_{L^{\frac{2q}{q-2}}} & \le \| \nabla Z \|_{L^2}\| \nabla Z \|_{L^q} \le C \| \nabla Z \|_{L^2} \| \nabla Z \|_{H^1} \\ &  \le C(\epsilon) \| \nabla Z \|^2_{L^2} + \epsilon \| Z \|^2_{H^2}.
\end{split}
\end{equation}
\begin{equation}
\begin{split}
I_2= \int_{\Omega} k\rho \phi(\theta)Z |Z_t| & \le c  \int_{\Omega} \rho \sqrt{\theta} e^{-\frac{1}{\theta}}|Z_t| Z  \le \tilde{c} \| \sqrt{\rho}Z_t \|^2_{L^2}.
\end{split}
\end{equation}
In order to estimate $\| Z \|^2_{H^2}$ from the mass fraction equation we have
\begin{equation}
\Delta Z= \dfrac{1}{D} \rho[ Z_t+ u \nabla Z + k\phi(\theta)Z],
\end{equation}
from which we deduce
\begin{equation}
\| Z \|_{H^2} \le c( \| \sqrt{\rho}Z_t \|_{L^2} + \| \sqrt{\rho}u \nabla Z\|_{L^2}+ \| Z \|_{L^2}).
\end{equation}
Moreover,  since
$$ \| \sqrt{\rho}u \nabla Z\|^2_{L^2} \le C( \epsilon) \| \nabla Z \|^2_{L^2}+ \epsilon \| Z \|^2_{H^2}. $$
we deduce that 
\begin{equation}\label{new 2.18 fan}
\| Z \|_{H^2} \le \tilde{c} ( \| \sqrt{\rho}Z_t \|_{L^2}+ \| \nabla Z \|_{L^2}+ \| Z \|_{L^2}).
\end{equation}
Summing up \eqref{new 2.18 fan} with the estimates for the $I_j$ we get the following inequality
\begin{equation*}
\begin{split}
\dfrac{d}{dt}\int_{\Omega} \dfrac{D}{2}|\nabla Z|^2+\tilde{C} \int_{\Omega} \rho |Z_t|^2  & \le \| \nabla Z \|^2_{L^2}+ \| Z \|^2_{H^2}+ \| Z \|^2_{L^2} \\ & \le \| \nabla Z \|^2_{L^2}+ \| \sqrt{\rho}Z_t \|^2_{L^2}+ \| \nabla Z \|^2_{L^2}+ \| Z \|^2_{L^2}
\end{split}
\end{equation*}
and so we end up with
\begin{equation}
\begin{split}
\dfrac{d}{dt}\int_{\Omega} \dfrac{D}{2}|\nabla Z|^2+\overline{C} \int_{\Omega} \rho |Z_t|^2  & \le C_1  \| \nabla Z \|^2_{L^2}+ \| Z \|^2_{L^2}.
\end{split}
\end{equation}
By applying Gronwall inequality we have 
$$ \dfrac{D}{2}\| \nabla Z\|^2_{L^2} \le c + \| Z \|^2_{L^2} \exp \left( \int_{0}^{t} C_1ds \right) $$
from which it follows 
$$ \sup \limits_{0 \le t \le T} \| \nabla Z \|^2_{L^2} \le c,$$
or equivalentely
$$ \sup \limits_{0 \le t \le T} \|  Z \|^2_{H^1_0} \le c.$$
\end{proof}

\begin{lemma}
\label{Lemma 2.9 cap4}
For any $ T < T_*, $ the following estimates hold
\begin{equation}
\sup \limits_{0 \le t \le T} \| Z \|^2_{H^2}  \le C,
\end{equation}
\begin{equation}
\sup \limits_{0 \le t \le T} \| \sqrt{\rho}Z_t \|^2_{L^2} + \int_{0}^{T} \| Z_t \|^2_{H^1_0} dt \le c.
\end{equation}
\end{lemma}
\begin{proof}
We take the time derivative of the equation for $Z$ in \eqref{pb1} and we multiply it by $Z_t$.  Integrating over $\Omega$ we end up with 
\begin{equation}\label{Z5 cap 4}
\begin{split}
&  \dfrac{1}{2}\dfrac{d}{dt}\int_{\Omega}\ \rho Z^2_t+ D\int_{\Omega} |\nabla Z_t|^2  \le  \int_{\Omega} |\rho_t| |u| | \nabla Z| |Z_t|+\int_{\Omega} \rho |u_t|  | \nabla Z | | Z_t| \\ & +k\int_{\Omega} \phi_t(\theta) \rho Z |Z_t|+k\int_{\Omega}\phi(\theta)|\rho_t|Z |Z_t|  +k\int_{\Omega}\phi(\theta)\rho |Z_t|^2 \\ &=\sum_{j=1}^5 I_j.
\end{split}
\end{equation}
Now we estimate all the integrals $I_j$,  $j=1,....,5.$ 
\begin{equation}\label{I1 CAPIT 4}
\begin{split}
I_1 & =\int_{\Omega} |\rho_t| |u| | \nabla Z| |Z_t| \\ &  \le \int_{\Omega} \rho| \operatorname{div}u | |u| | \nabla Z| |Z_t|+ \int_{\Omega} |u| | \nabla \rho | |u| |\nabla Z| |Z_t|= I_{11}+ I_{12}.
\end{split}
\end{equation}
\begin{equation*}
\begin{split}
I_{11} & = \int_{\Omega} \rho| \operatorname{div}u | |u| | \nabla Z| |Z_t| \le \| \sqrt{\rho}Z_t \|_{L^2} \| \operatorname{div}u\|_{L^6}\|u\|_{L^6}\| \nabla Z \|_{L^6} \\ & \le C_1 \| Z \|^2_{H^2}+C_2\| \sqrt{\rho}Z_t \|_{L^2} \|u \|^2_{H^2} \|u \|^2_{H^1}.
\end{split}
\end{equation*}
\begin{equation*}
\begin{split}
I_{12} & = \int_{\Omega} |u| | \nabla \rho | |u| |\nabla Z| |Z_t| \le \| \nabla \rho \|_{L^2} \| u \|_{L^\infty} \| u \|_{L^6} \| \nabla Z \|_{L^6} \| Z_t \|_{L^6} \\ &  \le c \| Z_t \|^2_{H^1}+ \tilde{c} \| \nabla \rho \|^2_{L^2} \|u \|^2_{H^2} \| u \|^2_{H^1} \| \nabla Z \|^2_{H^1}.
\end{split}
\end{equation*}
\begin{equation*}
\begin{split}
I_2=\int_{\Omega} \rho |u_t|  | \nabla Z | | Z_t| = \int_{\Omega}[ \operatorname{div}( \rho u_t Z)- \operatorname{div}( \rho u_t)Z]Z_t=I_{21}+I_{22}.
\end{split}
\end{equation*}
\begin{equation*}
\begin{split}
I_{21} & =  \int_{\Omega} \operatorname{div}(\rho u_t Z)Z_t= \int_{\Omega} \operatorname{div}(\rho u_t Z Z_t)-\int_{\Omega} \rho u_t Z \nabla {Z_t} \\ &  \le C_1  \| \sqrt{\rho}u_t \|^2_{L^2}+ C_2 \| \nabla {Z_t} \|^2_{L^2}.
\end{split}
\end{equation*}
\begin{equation}\label{I2J CAP 4}
I_{22}=- \int_{\Omega} \operatorname{div}( \rho u_t)ZZ_t \le \int_{\Omega} \rho \operatorname{div}u_tZ Z_t+ \int_{\Omega} \nabla \rho u_t Z Z_t.
\end{equation}
For the first integral in the right hand side of  \eqref{I2J CAP 4} we have that
\begin{equation}
\begin{split}
\int_{\Omega} \rho \operatorname{div}u_tZ Z_t &  \le C_1 \| \nabla {u_t}\|^2_{L^2}+ C_2\| \sqrt{\rho}Z_t \|^2_{L^2},
\end{split}
\end{equation} 
while the second can be estimated as follows
\begin{equation*}
\begin{split}
\int_{\Omega} \nabla \rho  u_t Z Z_t & \le \dfrac{1}{2} \| \nabla \rho \|^2_{L^2} \| \nabla {u_t} \|^2_{L^2}+\dfrac{1}{2} \| \nabla {Z_t} \|^2_{L^2}.
\end{split}
\end{equation*}
Then
\begin{equation}
\begin{split}
I_3 & = k \int_{\Omega} \phi_t(\theta) \rho Z |Z_t| =k \int_{}^{} \theta_t e^{-{\frac{1}{\theta}}} \bigg[\dfrac{1}{2 \sqrt{\theta}}+ \dfrac{1}{\sqrt{\theta^3}} \bigg] \rho Z |Z_t| \\ & \le c \| Z \|_{L^\infty} \int_{\Omega} \rho | \theta_t| |Z_t| \le \dfrac{c}{2}\| \sqrt{\rho}\theta_t\|^2_{L^2}+\dfrac{1}{2}\| \sqrt{\rho}Z_t\|_{L^2}.
\end{split}
\end{equation}

\begin{equation}
\begin{split}
I_4=k\int_{\Omega}\phi(\theta)|\rho_t|Z |Z_t| & \le c \| \rho_t \|_{L^2}  \| \sqrt{\theta} \|_{L^4}\| Z \|_{L^\infty} \| Z_t\|_{L^4} \\ & \le C_1 \| \rho_t \|_{L^2}+ C_2 \| \nabla {Z_t}\|^2_{L^2}.
\end{split}
\end{equation}

\begin{equation}\label{I5 CAPIT 4}
I_5=k\int_{\Omega}\phi(\theta)\rho |Z_t|^2 \le c \| \sqrt{\rho}Z_t\|^2_{L^2}.
\end{equation}
Summing up \eqref{I1 CAPIT 4}-\eqref{I5 CAPIT 4} we get the following inequality
\begin{equation}\label{gron capitolo 4}
\begin{split}
& \dfrac{1}{2}\dfrac{d}{dt} \| \sqrt{\rho}Z_t\|^2_{L^2}+ \tilde{c} \| \nabla {Z_t}\|^2_{L^2} \\ & \le c \| \sqrt{\rho}Z_t\|^2_{L^2}+ C_1 \| Z \|^2_{L^2}+ C_2 \| \sqrt{\rho}\theta_t\|^2_{L^2} \\ & + c\| \nabla \rho \|^2_{L^2} \| \nabla {u_t}\|^2_{L^2}+ \| u \|^2_{H^2}+ \| Z \|^2_{H^2},
\end{split}
\end{equation}
where we observe that all the terms in  the right hand side of \eqref{gron capitolo 4} have finite $L^1$ norm with respect to time.
Hence,  integrating \eqref{gron capitolo 4}  in time,  we get 
$$ \int_{0}^{T} \| Z_t \|^2_{H^1_0} \le C .$$
Moreover,  by applying Gronwall inequality we have
\begin{equation}
\| \sqrt{\rho}Z_t\|^2_{L^2} \le c+ F(t)
\end{equation}
with $F(t)$ properly defined such that the following estimate holds
$$ \sup \limits_{0 \le t \le T} \| \sqrt{\rho}Z_t \|^2_{L^2} \le C .$$
\end{proof}
We complete our analysis with the next Lemma. 
\begin{lemma}
\label{Lemma 2.10 cap4}
For any $ T < T_*, $ the following estimate holds
\begin{equation}
\int_{0}^{T} \|Z \|^2_{W^{2,q}} dt \le c.
\end{equation}
\end{lemma}

\begin{proof}
By the equation for $Z$ in \eqref{pb1} it follows
\begin{equation*}
\Delta Z= \dfrac{1}{D}[\rho Z_t + \rho u \nabla Z+k\phi(\theta)Z] 
\end{equation*}
from which we deduce 
\begin{equation}
\begin{split}
\| Z \|_{W^{2,q}} & \le c ( \| Z_t \|_{L^q}+ \| u \nabla Z \|_{L^q}+ \| Z \|_{L^q}) \\ &
\le \tilde{c}( \| \nabla {Z_t}\|_{L^2}+ \|u \|_{L^\infty}\| \nabla Z \|_{L^q}+ \| \nabla Z \|_{L^2}) \\ & \le c( \| \nabla {Z_t}\|_{L^2}+ \| u \|_{H^2} \| Z \|_{H^2}+ \| \nabla Z\|_{L^2}) 
\end{split}
\end{equation}
Integrating in time we get $$\int_{0}^{T} \|Z \|_{W^{2,q}} dt \le c.$$
\end{proof}
By using all the Lemmas,  \ref{Lemma 2.1 cap4} - \ref{Lemma 2.10 cap4}  we obtain bounds of the norms in the class of regularity \eqref{REGULARITY NL} of the solution $(\rho, \theta, u,Z)$. This estimates holds for any $T < T_*$ and depend only on $\Omega$,  on the initial data and on $T_*$ continuously (this time dependence is either polynomial or exponential). Hence,  we can consider a new initial data at $t=T$ and apply the local existence Theorem \ref{Teo2.1} in order to extend the solution.  Iterating this procedure,  we have that at a certain step,  there exists a time $T_1$ from which the solution does exists and such that $T+ T_1 > T_*$ which is a contradiction with respect to the maximality of $T_*$. This completes the proof of Theorem \ref{scndblowup}.

\end{document}